\documentclass{amsart}
\usepackage[margin=1in]{geometry}
\usepackage{amsmath}
\usepackage{amssymb}
\usepackage{amsthm}
\usepackage{amscd}
\usepackage{enumerate}
\newcounter{pcounter}

\newcommand{\ZZ}{\Bbb Z}
\newcommand{\RR}{\Bbb R}
\newcommand{\TT}{\Bbb T}
\newcommand{\NN}{\Bbb N}

\newcommand{\CC}{\Bbb C}
\newcommand{\ip}[1]{\langle #1 \rangle}

\newcommand{\widetidle}{\widetilde}

\newcommand{\varespilon}{\varepsilon}

\newcommand{\actson}{\curvearrowright}

\newtheorem{?}{Question}
\newtheorem{theorem}{Theorem}
\newtheorem{definition}[theorem]{Definition}
\newtheorem{proposition}[theorem]{Proposition}
\newtheorem{cor}[theorem]{Corollary}
\newtheorem{lemma}[theorem]{Lemma}

\newcommand{\FF}{\Bbb F}

\DeclareMathOperator{\Sym}{Sym}

\DeclareMathOperator{\Span}{Span}

\DeclareMathOperator{\Map}{Map}

\DeclareMathOperator{\tr}{tr}
\DeclareMathOperator{\supp}{supp}

\DeclareMathOperator{\Rea}{Re}

\DeclareMathOperator{\Hom}{Hom}

\DeclareMathOperator{\Tr}{Tr}

\DeclareMathOperator{\AP}{AP}
\DeclareMathOperator{\Prob}{Prob}
\DeclareMathOperator{\Ball}{Ball}

\numberwithin{theorem}{section}

\begin{document}
\title{Sofic Entropy of Gaussian Actions}      
\author{Ben Hayes}
\address{Stevenson Center\\
         Nashville, TN 37240}
\email{benjamin.r.hayes@vanderbilt.edu}
\date{\today}

\begin{abstract} Associated to any orthogonal representation of a countable discrete group is a probability measure-preserving action called the Gaussian action. Using the Polish model formalism we developed before, we compute the entropy (in the sense of Bowen, Kerr-Li) of Gaussian actions when the group is sofic. Computations of entropy for Gaussian actions has only been done when the acting group is abelian and thus our results are new even in the amenable case. Fundamental to our approach are methods of noncommutative harmonic analysis and $C^{*}$-algebras which replace the Fourier analysis used in the abelian case.

\end{abstract}



\maketitle
\tableofcontents

\section{Introduction}
	
	This paper is concerned with giving new computations for sofic entropy, specifically in computing entropy of Gaussian actions. Entropy for actions of $\ZZ$ is classical and goes back to work of Kolmogorov and Sina\v\i. Entropy roughly measures how chaotic the action of $\ZZ$ is.  Kieffer in \cite{Kieff} showed that one can generalize entropy to actions of amenable groups. An amenable group is a group which has a sequence of nonempty finite sets which are almost invariant under translation by elements of the group. Abelian groups, nilpotent groups and solvable groups are amenable, but the free group on $r$ letters is not if $r\geq 2.$ Entropy for amenable groups has been studied by many people and is a useful invariant in ergodic theory: it can be computed in many cases and positivity of entropy implies interesting structural properties.
	
	Fundamental examples in \cite{OrnWeiss} led many to believe that there cannot be a good entropy theory beyond the realm of amenable groups. In groundbreaking and seminal work Bowen in \cite{Bow} defined a notion of entropy for actions of sofic groups. The class of sofic groups is considerably larger than that of amenable groups: it contains all amenable groups, all residually finite groups, all linear groups  and is closed under free products with amalgamation over amenable subgroups (see \cite{DKP},\cite{KerrDykemaPichot2},\cite{ESZ1},\cite{LPaun},\cite{PoppArg}). Since the subject is fairly young, not as much is known about sofic entropy as entropy for actions of amenable groups but some structure is beginning to emerge. It can be calculated for some interesting examples such as Bernoulli shifts (see \cite{Bow},\cite{KLi}) as well as algebraic actions (see \cite{BowenEntropy},\cite{BowenLi},\cite{Me5},\cite{KLi}). Additionally, recent work in \cite{Meyerovitch},\cite{Me6} shows that one can deduce structural properties of an action from assumptions of positive sofic entropy.

Our goal in this paper is to add to the list of computations of sofic entropy by computing the entropy of \emph{Gaussian actions}.  The Gaussian  action construction is a way to associate, in a functorial way, to any orthogonal representation $\rho\colon\Gamma\to \mathcal{O}(\mathcal{H})$ a probability measure-preserving action called the Gaussian action.  We refer the reader to Section \ref{S:Gauss} for the precise definition.  For intuition we mention that if $\mathcal{H}$ is finite-dimensional, then this action is just the induced action of $\Gamma$ on $\mathcal{H}$ with the Gaussian measure. The Gaussian action  is a natural generalization of this construction to the case of infinite-dimensional representations.

	Recall that the left regular representation $\lambda_{\RR}\colon\Gamma\to \ell^{2}(\Gamma,\RR)$ is given by
	\[(\lambda_{\RR}(g)\xi)(h)=\xi(g^{-1}h),\mbox{ for $\xi\in \ell^{2}(\Gamma,\RR),g,h\in\Gamma.$}\]
It is known that the Gaussian action associated to the left-regular representation $\Gamma\actson \ell^{2}(\Gamma,\RR)$ is the Bernoulli action with base $(\RR,\nu)$ where $\nu$ is the Gaussian measure.  Thus Gaussian actions are a class of actions which are similar to the class of Bernoulli shifts.

 	To the best of our knowledge, entropy for Gaussian actions has only been computed when the acting group is abelian (see \cite{Leman} and \cite{Dooley}). To state our result, we need to introduce the following decomposition of representations. We define singularity of orthogonal representations exactly as in the unitary case. Given orthogonal representations $\rho_{1},\rho_{2}$ of $\Gamma$, we say that $\rho_{1}$ is singular with respect to $\rho_{2}$ and write $\rho_{1}\perp \rho_{2}$ if no nonzero subrepresentation of $\rho_{1}$ embeds into $\rho_{2}.$ We say that $\rho_{1}$ is \emph{absolutely continuous with respect to} $\rho_{2}$, and write $\rho_{1}\ll \rho_{2},$ if $\rho_{1}$ embeds into $\rho_{2}^{\oplus \infty}.$ For general pairs of orthogonal representations $\rho_{1},\rho_{2}$ we can always write
 \[\rho_{1}=\rho_{1,a}\oplus \rho_{1,s}\]
 where $\rho_{1,s},\rho_{1}$ are singular and $\rho_{1,a}\ll \rho_{2}.$ If $\Gamma$ is abelian, this reduces to the Lebesgue decomposition in measure theory. We thus regard this as a noncommutative analogue of the Lebesgue decomposition.

\begin{theorem}\label{T:GaussianActions} Let $\Gamma$ be a countable discrete sofic group with sofic approximation $\Sigma.$ Let $\rho\colon \Gamma\to \mathcal{O}(\mathcal{H})$ be an orthogonal representation with $\mathcal{H}$ separable. Let $\Gamma\actson (X_{\rho},\mu_{\rho})$ be the induced Gaussian action; write $\rho=\rho_{a}\oplus \rho_{s}$ where $\rho_{s}$ and $\lambda_{\Gamma,\RR}$ are singular and $\rho_{a}\ll\lambda_{\Gamma,\RR}.$ Then
\[h_{\Sigma,\mu_{\rho}}(X_{\rho},\Gamma)=\begin{cases}
-\infty, &\textnormal{ if $h_{\Sigma,\mu_{\rho_{s}}}(X_{\rho_{s}},\Gamma)=-\infty$}\\
0,      &\textnormal{if $\rho_{a}=0$ and $h_{\Sigma,\mu_{\rho_{s}}}(X_{\rho_{s}},\Gamma)\ne -\infty$}\\
\infty, &\textnormal{if $\rho_{a}\ne 0$ and $h_{\Sigma,\mu_{\rho_{s}}}(X_{\rho_{s}},\Gamma)\ne -\infty$.}
\end{cases}\]
\end{theorem}
Since it appears to be new, we specifically mention the amenable case.
\begin{cor} Let $\Gamma$ be a countable discrete amenable group. Let $\rho\colon \Gamma\to \mathcal{O}(\mathcal{H})$ be an orthogonal representation and $\Gamma\actson (X_{\rho},\mu_{\rho})$ the induced Gaussian action. Write $\rho=\rho_{1}\oplus \rho_{2}$ where $\rho_{2}$ is embeddable into $\Gamma\actson \ell^{2}(\Gamma,\RR)^{\oplus\infty},$ and $\Hom_{\Gamma}(\rho_{1},\lambda_{\Gamma,\RR})=\{0\}.$ Then $h_{\mu_{\rho}}(X_{\rho},\Gamma)=\infty$ if and only if $\rho_{2}\ne 0,$ and if $\rho_{2}=0,$ then $h_{\mu_{\rho}}(X_{\rho},\Gamma)=0.$
\end{cor}

In \cite{Me5} we gave a formula for entropy of a probability measure-preserving action $\Gamma\actson (X,\mu)$ when $X$ is a Polish space, $\mu$ is a Borel probability measure, and $\Gamma\actson X$ by homeomorphisms. Given an arbitrary probability measure-preserving action $\Gamma\actson (Y,\nu),$ a probability measure-preserving action $\Gamma\actson (X,\mu)$ where $X$ is Polish, $\mu$ is a Borel measure on $X$ and $\Gamma\actson X$ by homeomorphisms is called a Polish model for $\Gamma\actson (Y,\nu)$ if $\Gamma\actson (X,\mu)\cong \Gamma\actson (Y,\nu)$ as probability measure-preserving actions. Our definition of entropy in terms of a Polish model took into account the topology of $X$ in a nontrivial way like the definition for entropy in the presence of a compact model developed by Kerr-Li in \cite{KLi}. Our computation for the entropy of Gaussian actions goes through Polish models. This model is associated to a family of generators for the representation, and the measure is canonically defined in terms of the representation. Although one can write down a compact model, it is unnatural and the measure is not expressed nicely in terms of the representation. We mention that entropy in the sofic case is roughly a measure of ``how many'' finitary simulations a probability measure-preserving action has. The typical way to prove existence of these simulations is through a probabilistic argument. For the Gaussian action, our probabilistic argument uses  Gaussian measures on finite-dimensional spaces. A consequence of our methods is that, associated to any sofic approximation of a group, we have a natural way of describing the Gaussian measure for a subrepresentation of the left regular representation as a limit of finite-dimensional Gaussian measures.
	
	Let us sketch how we are able to handle the case $\rho$ is singular with respect to $\lambda_{\RR}$ when $\Gamma$ is nonabelian. In the case $\Gamma=\ZZ$, the problem of showing that $\Gamma\actson (X_{\rho},\mu_{\rho})$ has zero entropy reduces to the fact that if $\mu,\nu$ are two singular measures on $\TT=\RR/\ZZ,$ then there is a $f\in C(\TT)$ with $0\leq f\leq 1$ so that $f$ is ``close'' to $1$ in $L^{2}(\mu)$ and close to zero in $L^{2}(\nu).$ As the representation theory of a group is captured by its universal $C^{*}$-algebra it is natural to replace $C(\TT)$ with $C^{*}(\Gamma)$ in the nonabelian case. This statement about singularity of representations then becomes a statement about singularity of measures and we can prove an analogous characterization of singularity of representations in terms of elements of $C^{*}(\Gamma).$ In short, we abstract the harmonic analysis present in the abelian case to noncommutative harmonic analysis using the framework of $C^{*}$-algebras.

We mention that if $\Gamma$ is amenable and is an infinite conjugacy class group (i.e. every nontrivial conjugacy class is infinite) then there is a easy proof that if $\rho\colon \Gamma\to\mathcal{O}(\mathcal{H})$ embeds into $\lambda_{\Gamma,\RR}^{\oplus \infty},$ then $h_{\mu_{\rho}}(X_{\rho},\Gamma)=\infty.$ Namely, one can show in this case that there is an $n$ so that $\rho^{\oplus n}$ contains $\lambda_{\Gamma,\RR}$ (this is a consequence of the theory of $\textrm{II}_{1}$ factors).  In this case $(X_{\rho},\mu_{\rho})$ will factor onto  $(X_{\lambda},\mu_{\lambda})$ which is isomorphic to $\Gamma\actson (\RR,\eta)^{\Gamma}$ where $\eta$ is the Gaussian measure. Since entropy for \emph{amenable} groups decreases under factor maps, we are done. This proof fails disastrously in the non-amenable case. It is  \emph{very} far from true that sofic entropy decreases under factor maps for non-amenable groups. In fact, it can be shown that for \emph{every} nonamenable group there is a $\alpha>0$ so that if $(X,\mu)$ is a measure space with $H(X,\mu)\geq \alpha$ then $\Gamma\actson (X,\mu)^{\Gamma}$ factors onto \emph{every} nontrivial Bernoulli shift (see \cite{BowenOrn} Corollary 1.6), even a Bernoulli shift with infinite entropy.  Moreover,  if $\Gamma$ contains a free group then $\alpha$ can be taken to be any positive number \cite{BowenWeak}. In fact, even more is true: a recent result of Seward (see \cite{SewardSmallFactor}, Theorem 1.1) implies that for every nonamenable sofic group $\Gamma,$ there is an $\alpha>0$ so that any probability measure-preserving action of  $\Gamma$ is a factor of an action having sofic entropy less than $\alpha.$  Thus, there is no simple proof in the nonamenable case based on factors, and we \emph{must} use a direct proof. Even in the amenable case this argument relies on the group being an infinite conjugacy class group and one needs general methods to handle the general case.

\section{Preliminaries on Sofic Entropy}

We use $S_{n}$ for the symmetric group on $n$ letters. If $A$ is a  set, we will use $\Sym(A)$ for the set of bijections of $A.$
\begin{definition}\label{S:soficdefn}\emph{Let $\Gamma$ be a countable discrete group. A} sofic approximation of $\Gamma$ \emph{is a sequence $\Sigma=(\sigma_{i}\colon \Gamma\to S_{d_{i}})$ of functions (not assumed to be homomorphisms) so that}
\[\lim_{i\to\infty}u_{d_{i}}(\{1\leq k\leq d_{i}:\sigma_{i}(g)\sigma_{i}(h)(k)=\sigma_{i}(gh)(k)\})=1,\mbox{\emph{for all $g,h\in \Gamma$}}\]
\[\lim_{i\to\infty} u_{d_{i}}(\{1\leq k\leq d_{i}:\sigma_{i}(g)(k)\ne \sigma_{i}(h)(k)\})=1,\mbox{\emph{for all $g\ne h$ in $\Gamma.$}}\]
\emph{We will call $\Gamma$} sofic \emph{if it has a sofic approximation.}\end{definition}
 All amenable groups and residually finite groups are sofic. Also, it is known that soficity is closed under free products  with amalgamation over amenable subgroups (see \cite{ESZ1},\cite{LPaun},\cite{DKP},\cite{KerrDykemaPichot2}, \cite{PoppArg}). Additionally, residually sofic groups and locally sofic groups are sofic. Thus by Malcev's Theorem we know all linear groups are sofic.  It is known that graph products of sofic groups are sofic by \cite{CHR}. If $\Lambda$ is a subgroup of $\Gamma,$ and  $\Lambda$ is sofic and $\Gamma\actson \Gamma/\Lambda$ is amenable (in the sense of having a $\Gamma$-invariant mean) then $\Gamma$ is sofic. One can argue this by the same methods of Theorem 1 of \cite{ESZ1} (for example consider the observations after Definition 12.2.12 of \cite{BO}).

 	We recall the definition of entropy in the presence of a Polish model given in \cite{Me6}. Let $X$ be a Polish space and let $\Gamma$ be a countable, discrete group with $\Gamma\actson X$ by homeomorphisms. We say that a continuous  pseudometric $\Delta$ on $X$ is \emph{dynamically generating} if for every $x\in X$ and every neighborhood $U$ of $x$ in $X$ there exists a finite $F\subseteq \Gamma$ and a $\delta>0$ so that if $y\in X$ and
\[\max_{g\in F}\Delta(gx,gy)<\delta\]
then $y\in U.$ Notice that our definition includes the hypothesis that $\Delta$ is continuous. We use $C_{b}(X)$ for the Banach space of bounded, continuous functions on $X$ with norm
\[\|f\|=\sup_{x\in X}|f(x)|.\]

	Given a pseudometric space $(X,\Delta),$ $A,B\subseteq X$ and $\varepsilon>0,$  we say that $A$ is $\varepsilon$-contained in $B$ and write $A\subseteq_{\varepsilon}B$ if for any $a\in A,$ there is a $b\in B$ so that $\Delta(a,b)\leq\varepsilon.$ We say that $A\subseteq X$ is $\varepsilon$-dense if $X\subseteq_{\varepsilon}A.$ We use $S_{\varepsilon}(X,\Delta)$ for the minimal cardinality of an $\varepsilon$-dense subset of $X.$ We say that $A\subseteq X$ is $\varepsilon$-separated if for any $a_{1}\ne a_{2}$ in $A$ we have $\Delta(a_{1},a_{2})>\varepsilon.$ We use $N_{\varepsilon}(A,\Delta)$ for the maximal cardinality of an $\varepsilon$-separated subset of $X.$ We leave it as an exercise to show that
	\begin{equation}\label{E:separationspanning}
N_{2\varepsilon}(A,\Delta)\leq S_{\varepsilon}(A,\Delta)\leq N_{\varepsilon}(A,\Delta),
\end{equation}
and that if $\delta>0$ and $A\subseteq_{\delta}B$ then
\[S_{2(\varepsilon+\delta)}(A,\Delta)\leq S_{\varepsilon}(B,\Delta).\]
We use $\Delta_{2}$ for the metric on $X^{n}$ defined by
\[\Delta_{2}(\phi,\psi)^{2}=\frac{1}{n}\sum_{j=1}^{n}\Delta(\phi(j),\psi(j))^{2}.\]
If $X$ is a Polish space we use $\Prob(X)$ for the space of all Borel probability measures on $X.$  If $\Gamma$ is a countable, discrete group with $\Gamma\actson X$ by homeomorphisms we use $\Prob_{\Gamma}(X)$ for the space of all $\Gamma$-invariant elements of $\Prob(X).$  We are now ready to state our definition of sofic entropy from \cite{Me6}. It is defined by counting the exponential growth of maps from $\{1,\dots,d_{i}\}\to X$ which approximately preserve the measure-theoretic structure of $X$ and are approximately equivariant. We call these maps \emph{microstates} (this is a heuristic term and will not be defined rigorously).

\begin{definition}\emph{Let $\Gamma$ be a countable, discrete, sofic group with sofic approximation $\Sigma=(\sigma_{i}\colon\Gamma\to S_{d_{i}}).$ Let $X$ be Polish space with $\Gamma\actson X$ by homeomorphisms. Fix a bounded, continuous pseudometric on $X.$ For a finite $F\subseteq\Gamma$ and $\delta>0$ we let $\Map(\Delta,F,\delta,\sigma_{i})$ be all $\phi\in X^{d_{i}}$ so that}
\[\max_{g\in F}\Delta_{2}(g\phi,\phi\circ \sigma_{i}(g))<\delta.\]
\end{definition}

Notice that $\Map(\Delta,F,\delta,\sigma_{i})$ only accounts for the group action and not the measure-theoretic structure of $X.$ Recall that $C_{b}(X)$ denotes the space of bounded continuous functions on $X.$ For $\mu\in \Prob(X),$ a finite $L\subseteq C_{b}(X),$ and a $\delta>0$ we let
\[U_{L,\delta}(\mu)=\bigcap_{f\in L}\left\{\nu\in \Prob(X):\left|\int f\,d\mu-\int f\,d\nu\right|<\delta\right\}.\]
The sets $U_{L,\delta}(\mu)$ form a basis of open sets for a topology called the \emph{weak topology}. We use this topology to account for the measure-theoretic structure in our microstates.

\begin{definition}\emph{Let $\Gamma$ be a countable, discrete, sofic group with sofic approximation $\Sigma=(\sigma_{i}\colon\Gamma\to S_{d_{i}}).$ Let $X$ be Polish space with $\Gamma\actson X$ by homeomorphisms and fix $\mu\in\Prob_{\Gamma}(X).$ For finite $F\subseteq\Gamma,L\subseteq C_{b}(X)$ and $\delta>0$ we set $\Map_{\mu}(\Delta,F,\delta,L,\sigma_{i})$ to be the set of all $\phi\in \Map(\Delta,F,\delta,L,\sigma_{i})$ so that $\phi_{*}(u_{d_{i}})\in U_{L,\delta}(\mu).$}
\end{definition}

\begin{definition}\emph{Let $\Gamma$ be a countable, discrete, sofic group with sofic approximation $\Sigma=(\sigma_{i}\colon\Gamma\to S_{d_{i}}).$ Let $X$ be Polish space with $\Gamma\actson X$ by homeomorphisms and let $\mu\in \Prob_{\Gamma}(X).$ Fix a bounded, dynamically generating pseudometric on $X.$ For finite $F\subseteq\Gamma,L\subseteq C_{b}(X)$ and $\delta,\varespilon>0$ we set}
\[h_{\Sigma,\mu}(\Delta,F,\delta,L,\varepsilon)=\limsup_{i\to\infty}\frac{1}{d_{i}}\log N_{\varepsilon}(\Map_{\mu}(\Delta,F,\delta,L,\sigma_{i}),\Delta_{2}),\]
\[h_{\Sigma,\mu}(\Delta,\varepsilon)=\inf_{\substack{F\subseteq\Gamma\mbox{ finite},\\ L\subseteq C_{b}(X)\mbox{ finite},\delta>0}}h_{\Sigma,\mu}(\Delta,F,\delta,L,\varepsilon),\]
\[h_{\Sigma,\mu}(X,\Gamma)=\sup_{\varepsilon>0}h_{\Sigma,\mu}(\Delta,\varepsilon).\]
\emph{We call $h_{\Sigma,\mu}(X,\Gamma)$ the entropy of $\Gamma\actson (X,\mu)$ with respect to $\Sigma.$}
\end{definition}

It is shown in \cite{Me6} that this agrees with entropy as defined by \cite{Bow},\cite{KLi}.

\section{Generating Sets and Tightness}\label{S:tightness}
Since $C_{b}(X)$ is not separable, we would like to reduce checking the approximate measure-preserving property of our microstates from all functions in $C_{b}(X)$ to a smaller class of functions. For example $C_{b}(X)$ is separable in the topology of uniform convergence on compact sets, so if we require that we have a family of functions dense in this topology then this will give us a sufficiently small family of functions to deal with. However, for this to work we need to modify our microstates so that they have some uniform tightness. We proceed with the definitions.

\begin{definition}\label{D:generating}\emph{ Let $X$ be a Polish space. A family $\mathcal{L}\subseteq C_{b}(X)$ is said to be} generating \emph{if there is a $A>0$ so that for every $g\in C_{b}(X),$ for every compact $K\subseteq X$ and for every $\varepsilon>0$ there is a $f\in \Span(\mathcal{L})$ so that $\|f\big|_{K}-g\big|_{K}\|<\varepsilon$ and $\|f\|\leq  A\|g\|.$}
\end{definition}

We now proceed with our modified version  of sofic entropy in the case of a generating set of functions.

\begin{definition}\emph{Let $\Gamma$ be a countable discrete group and $\sigma\colon\Gamma\to S_{d}$ a function for some $d\in\NN.$ Let $X$ be a Polish space with $\Gamma\actson X$ by homeomorphisms preserving a Borel probability measure $\mu.$ Let $\Delta$ be a dynamically generating pseudometric on $X.$  For  an open subset $U$ of $X,$ for  $\eta,\delta>0$ and finite $L\subseteq C_{b}(X),F\subseteq \Gamma$ we let $\Map^{U,\eta}_{\mu}(\Delta,F,\delta,L,\sigma)$ be the set of
$\phi\in\Map_{\mu}(\Delta,F,\delta,L,\sigma)$ so that}
\[\phi_{*}(u_{d})(U)>1-\eta.\]
\emph{Suppose now that $\Gamma$ is sofic. For a dynamically generating pseudometric $\Delta,$  $\mathcal{L}\subseteq C_{b}(X)$  and a sofic approximation $\Sigma=(\sigma_{i}\colon\Gamma\to S_{d_{i}})$ set:}
\[h_{\Sigma,\mu}^{U,\eta}(\Delta,\varepsilon,F,\delta,L)=\limsup_{i\to \infty}\frac{1}{d_{i}}\log N_{\varepsilon}(\Map_{\mu}^{U,\eta}(\Delta,F,\delta,L,\sigma_{i})),\]
\[h_{\Sigma,\mu}^{K,\eta}(\Delta,\varepsilon,F,\delta,L)=\inf_{\textnormal{open $U\supseteq K$}}h_{\Sigma,\mu}^{U,\eta}(\Delta,\varepsilon,F,\delta,L),\]
\[h_{\Sigma,\mu}^{\eta}(\Delta,\varepsilon,F,\delta,L)=\sup_{K\subseteq X\textnormal{ compact}}h_{\Sigma}^{K,\eta}(\Delta,\varepsilon,F,\delta,L)\]
\[h_{\Sigma,\mu}(\Delta,\varepsilon,\mathcal{L})=\inf_{\substack{\eta,\delta>0,\\ F\subseteq \Gamma \textnormal{ finite},\\ L\subseteq \mathcal{L}\textnormal{ finite}}}h_{\Sigma,\mu}^{\eta}(\Delta,\varepsilon,F,\delta,L)\]
\[h_{\Sigma,\mu}(\Delta,\mathcal{L})=\sup_{\varepsilon>0}h_{\Sigma,\mu}(\Delta,\varepsilon,\mathcal{L}).\]

\end{definition}

In the above definition, the necessary trick is to add more quantifiers to the definition of sofic entropy. The reader may be concerned already by the number of quantifiers involved in the original definition of sofic entropy. When we compute sofic entropy of Gaussian actions in Section \ref{S:Gauss} it will be clear that this is the correct tradeoff. The difficulty involved in the computation will not be in dealing with quantifiers  but instead that we can only show the approximate measure-preserving property on a class of functions which is not norm dense. However, one can easily see that this class of functions is generating and so we will use the above definition.

\begin{theorem}\label{T:generating} Let $\Gamma$ be a countable discrete sofic group with sofic approximation $\Sigma.$ Let $X$ be a Polish space with $\Gamma\actson X$ by homeomorphisms preserving a Borel probability measure $\mu.$ For any dynamically generating pseudometric $\Delta$ on $X,$ $\mathcal{L}\subseteq C_{b}(X)$ generating, and $\varepsilon>0$
\[h_{\Sigma,\mu}(\Delta,\varepsilon)=h_{\Sigma,\mu}(\Delta,\varepsilon,\mathcal{L}).\]
In particular,
\[h_{\Sigma,\mu}(X,\Gamma)=h_{\Sigma,\mu}(\Delta,\mathcal{L}).\]
\end{theorem}

\begin{proof} The ``in particular'' part follows from Theorem 3.12 of \cite{Me6}. We first show that
\[h_{\Sigma,\mu}(\Delta,\varepsilon)\leq h_{\Sigma,\mu}(\Delta,\varepsilon,\mathcal{L}).\]
Let $\eta>0.$ Since $X$ is Polish, we may find a compact $K\subseteq X$ with
\[\mu(X\setminus K)<\eta.\]
Fix an open subset $U$ of $X$ containing $K.$
 By Urysohn's Lemma we may find a $f\in C_{b}(X)$ with
\[\chi_{K}\leq f\leq \chi_{U}.\]
Note that if $\nu\in\Prob(X)$ and $\delta>0$ with
\[\left|\int_{X}f\,d\mu-\int_{X}f\,d\nu\right|<\delta,\]
we have
\[\nu(U)>1-\eta-\delta.\]
We now see that for all finite $L\subseteq\mathcal{L},F\subseteq\Gamma,$ $0<\delta<\eta$ we have
\[\Map^{U,\eta+\delta}(\Delta,F,\delta,L,\sigma_{i})\supseteq \Map(\Delta,F,\delta,L\cup\{f\},\sigma_{i}).\]
Thus for all $\delta<\eta,$ and all  finite $F'\supseteq F,L'\supseteq L\cup\{f\}<$ and $\delta'<\delta,$
\[ h_{\Sigma,\mu}(\Delta,\varepsilon,F',\delta',L')\leq h_{\Sigma,\mu}^{U,2\eta}(\Delta,\varepsilon,F,\delta,L).\]
Taking the infimum over $F',\delta',L'$ we find
\[h_{\Sigma,\mu}(\Delta,\varepsilon)\leq h_{\Sigma,\mu}^{U,2\eta}(\Delta,\varepsilon,F,\delta,L).\]
Taking the infimum over all $U\supseteq K$ we find that
\[h_{\Sigma,\mu}(\Delta,\varepsilon)\leq h_{\Sigma,\mu}^{K,2\eta}(\Delta,\varepsilon,F,\delta,L).\]
Now take the supremum over all $K$ and then the infimum over all $\eta,\delta,F,L\subseteq \mathcal{L}$ to find that
\[h_{\Sigma,\mu}(\Delta,\varepsilon)\leq h_{\Sigma,\mu}(\Delta,\varepsilon,\mathcal{L}) .\]

We now show that
\[h_{\Sigma,\mu}(\Delta,\varepsilon,\mathcal{L})\leq h_{\Sigma,\mu}(\Delta,\varepsilon).\]
Let $A>0$ be as in the definition of generating. Let $F\subseteq \Gamma$ and $L\subseteq C_{b}(X)$ be given finite sets and $\delta>0.$ Let $\kappa>0$ be sufficiently small depending upon $F,L,\delta$ in a manner to be determined later. Since $X$ is Polish, we may find a compact $K_{0}\subseteq X$ so that
\[\mu(K_{0})\geq 1-\kappa.\]
 Let $K\subseteq X$ compact be given with $K\supseteq K_{0}.$   Let $L=\{f_{1},\dots,f_{l}\}.$ For $1\leq j\leq l$ we can find a $g_{j}\in \Span(\mathcal{L})$ so that
\[\|g_{j}\big|_{K}-f_{j}\big|_{K}\|<\delta\]
and
\[\|g_{j}\|\leq A\|f_{j}\|.\]
We may find an open neighborhood $U_{0}$ of $K$ so that
\[\|g_{j}\big|_{U_{0}}-f_{j}\big|_{U_{0}}\|_{C_{b}(U_{0})}<\delta, \mbox{ for $j=1,\dots,l$.}\]
Let $\delta'>0$ and $F'\subseteq\Gamma$ and $L'\subseteq C_{b}(X)$ be finite sets which will depend upon $F,\delta,L$ in a manner to be determined shortly. Let $U$ be an open neighborhood of $K$ contained in $U_{0}.$ For $\phi\in\Map_{\mu}^{U,\eta}(\Delta,F',\delta',L',\sigma_{i})$ we have
\begin{align*}
\left|\int_{X}f_{j}\,d\mu-\int_{X}f_{j}\,d\phi_{*}(u_{d_{i}})\right|&\leq \left|\int_{X}f_{j}\,d\phi_{*}(u_{d_{i}})-\int_{X}g_{j}\,d\phi_{*}(u_{d_{i}})\right|+\left|\int_{X}g_{j}\,d\mu-\int_{X}g_{j}\,d\phi_{*}(u_{d_{i}})\right|\\
&+\left|\int_{X}g_{j}\,d\mu-\int_{X}f_{j}\,d\mu\right|\\
&\leq \eta(1+A)\|f_{j}\|+\kappa(1+A)\|f_{j}\|+\left|\int_{U}f_{j}\,d\phi_{*}(u_{d_{i}})-\int_{U}g_{j}\,d\phi_{*}(u_{d_{i}})\right|\\
&+\left|\int_{X}g_{j}\,d\mu-\int_{X}g_{j}\,d\phi_{*}(u_{d_{i}})\right|+\left|\int_{U}g_{j}\,d\mu-\int_{U}f_{j}\,d\mu\right|\\
&\leq \eta(1+A)\|f_{j}\|+\kappa(1+A)\|f_{j}\|+2\delta+\left|\int_{X}g_{j}\,d\mu-\int_{X}g_{j}\,d\phi_{*}(u_{d_{i}})\right|\\
\end{align*}

If we choose $\delta'>0$ sufficiently small, and $F',L'$ sufficiently large we may force
\[\left|\int_{X}g_{j}\,d\mu-\int_{X}g_{j}\,d\phi_{*}(u_{d_{i}})\right|<\delta.\]
If we choose
\[\max_{j}2\eta(1+A)\|f_{j}\|<\delta\]
\[\max_{j}\kappa(1+A)\|f_{j}\|<\delta\]
then we have
\[\Map^{U,\eta}(\Delta,F',\delta',L',\sigma_{i})\subseteq \Map(\Delta,F,7\delta,L,\sigma_{i}).\]
Thus
\[h_{\Sigma,\mu}^{K,\eta}(\Delta,\varepsilon,F',\delta',L')\leq h_{\Sigma,\mu}^{U,\eta}(\Delta,\varepsilon,F',\delta',L')\leq h_{\Sigma,\mu}(\Delta,\varepsilon,F,7\delta,L).\]
Since $K\supseteq K_{0}$ was arbitrary, we may take the supremum over all $K$ to see that
\[h_{\Sigma,\mu}(\Delta,\varepsilon,\mathcal{L})\leq h_{\Sigma,\mu}^{\eta}(\Delta,\varepsilon,F',\delta',L')\leq h_{\Sigma,\mu}(\Delta,\varepsilon,F,7\delta,L).\]
Now taking the infimum over all $F',\delta',L'$ proves that
\[h_{\Sigma,\mu}(\Delta,\varepsilon,\mathcal{L})\leq h_{\Sigma,\mu}(\Delta,\varepsilon).\]

\end{proof}

\section{Entropy of Gaussian Actions}

Gaussian actions are a natural class of actions induced by orthogonal representations of a group. When the representation is the left regular representation, the Gaussian action is simply the Bernoulli action on $(\RR,\nu)^{\Gamma}$ where $\nu$ is the Gaussian measure. In \cite{Me6}, we prove some results which give structural results of  the Koopman representation $\Gamma\actson L^{2}(X,\mu)$ of a probability measure-preserving action $\Gamma\actson (X,\mu)$ of a sofic group $\Gamma,$ under the assumption that this action has positive entropy. For example, we could show that if $\Gamma\actson (X,\mu)$ has positive entropy then $\Gamma\actson L^{2}(X,\mu)$ must contain a ``piece" of the left regular representation (for more precise statements see \cite{Me6} Theorem 1.1 and Corollary 1.2). We will exploit the connections to representation theory to apply our spectral consequences of positive entropy in \cite{Me6} to Gaussian actions. We will also exploit the similarity to Bernoulli shifts to compute the entropy of Gaussian actions.

Let $\Gamma$ be a countable discrete group. Recall that an orthogonal representation of $\Gamma$ on a real Hilbert space $\mathcal{H}$ is a homomorphism $\rho\colon\Gamma\to \mathcal{O}(\mathcal{H})$ where $\mathcal{O}(\mathcal{H})$ is the group of orthogonal transformations of $\mathcal{H}$ (i.e. the set of $O\in B(\mathcal{H})$ so that $\ip{O\xi,O\eta}=\ip{\xi,\eta}$ for $\xi,\eta\in\mathcal{H}.$) We let $\mathcal{H}_{\CC}=\mathcal{H}\otimes_{\RR}\CC=\mathcal{H}+i\mathcal{H}$ be the complexification of $\mathcal{H}$ equipped with the unique sesquilinear inner product extending the one on $\mathcal{H}.$ We let $\rho_{\CC}\colon\Gamma\to \mathcal{U}(\mathcal{H}+i\mathcal{H})$ be the complexification of $\rho,$ i.e. $\rho_{\CC}(g)$ for $g\in\Gamma$ is the unique unitary transformation so that $\rho_{\CC}(g)(\xi)=\rho(g)\xi$ for $\xi\in\mathcal{H}.$

\subsection{Definition of Gaussian Actions}\label{S:Gauss}

The most natural way to define Gaussian actions is by von Neumann algebras.
\begin{definition}\emph{Let $\mathcal{H}$ be a complex Hilbert space. A von Neumann algebra on $\mathcal{H}$ is a $*$-subalgebra of $B(\mathcal{H})$ containing the identity and closed in the weak operator topology. We say that a vector $\xi\in \mathcal{H}$ is} cyclic \emph{ for $M$ if $\overline{M\xi}=\mathcal{H}.$}\end{definition}
Suppose that $(X,\mu)$ is a standard probability space. For $f\in L^{\infty}(X,\mu),$ let $M_{f}\in B(L^{2}(X,\mu))$ be given by
\[(M_{f}\xi)(x)=f(x)\xi(x).\]
The map $f\mapsto M_{f}$ allows us to view $L^{\infty}(X,\mu)$ as a von Neumann algebra. It turns out (see \cite{ConwayOT} Theorem 14.5)  that if $M\subseteq B(\mathcal{H})$ is a commutative von Neumann algebra with cyclic vector $\xi$ with $\|\xi\|_{2}=1$ and $\mathcal{H}$ is separable, then there is a standard probability space $(X,\mu)$ and a unitary
\[U\colon \mathcal{H}\to L^{2}(X,\mu)\]
so that
\[U(\xi)=1\]
\[UMU^{*}=\{M_{f}:f\in L^{\infty}(X,\mu)\}.\]
Additionally, if $\phi\colon M\to \CC$ is a linear functional so that
\[\phi\big|_{\{T\in M:\|T\|\leq 1\}}\]
is weak operator topology continuous, then there is a complex measure $\nu\ll\mu$ so that if $UTU^{*}=M_{f}$ then
\[\phi(T)=\int_{X}f\,d\nu.\]
We leave it to reader the verify that if $UTU^{*}=M_{f}$ then
\[\ip{T\xi,\xi}=\int_{X}f\,d\mu.\]

\begin{definition}\emph{Let $\mathcal{H}$ be a real Hilbert space. The Gaussian algebra associated to $\mathcal{H},$ denoted $\mathcal{A}(\mathcal{H})$ is a commutative von Neumann algebra with cyclic vector $\Omega$ with $\|\Omega\|=1,$ which is generated by unitaries $\{u(\xi):\xi\in\mathcal{H}\}$  satisfying}
\[u(\xi_{1}+\xi_{2})=u(\xi_{1})u(\xi_{2})\mbox{\emph{ for $\xi\in\mathcal{H}$}}\]
\[\ip{u(\xi)\Omega,\Omega}=\exp(-\pi\|\xi\|^{2}).\]
\emph{ For $a\in\mathcal{A}(\mathcal{H}),$ we let}
\[\phi(a)=\ip{a\Omega,\Omega}.\]
\emph{Suppose that $\Gamma$ is a countable discrete group and $\rho\colon\Gamma\to \mathcal{O}(\mathcal{H})$ is a representation. Then there is a $\phi$-preserving action $\alpha$ on $\mathcal{A}(\mathcal{H})$ determined uniquely by}
\[\alpha(g)(u(\xi))=u(\rho(g)\xi),\mbox{ for $\xi\in\mathcal{H}$.}\]
\emph{This action is called the Gaussian action.}\end{definition}
By \cite{PetersonSinclair} the Gaussian algebra exists and is unique up to $\phi$-preserving isomorphism. See also \cite{PetersonSinclair} for the existence and uniqueness of the Gaussian action.

	Let us sketch an alternate construction of the Gaussian algebra. First consider the case that $\mathcal{H}$ is finite-dimensional. Say $\mathcal{H}=\RR^{n}.$ In this case we can simply take $\mathcal{A}(\mathcal{H})=L^{\infty}(\RR^{n},\nu)$ where $d\nu=e^{-\pi\|x\|_{\ell^{2}(n)}^{2}}\,dx,$ and
\[\phi(a)=\int_{\RR^{n}}a(x)\,d\nu(x),\]
\[u(\xi)(x)=\exp(2\pi i \ip{\xi,x}).\]
 In general, consider the universal $*$-algebra $\widetilde{\mathcal{A}}(\mathcal{H})$ generated by unitaries $\widetilde{u}(\xi),\xi\in \mathcal{H}$ satisfying the relation $\widetilde{u}(\xi_{1}+\xi_{2})=\widetilde{u}(\xi_{1})\widetilde{u}(\xi_{2})$ for $\xi_{1},\xi_{2}\in\mathcal{H}.$ Here one can make sense of a unitary in a $*$-algebra by saying that it is an element $u$ so that $u^{*}u=uu^{*}=1.$ One then has to check that there is a well-defined linear function $\widetilde{\phi}\colon \widetilde{\mathcal{A}}(\mathcal{H})\to \CC$ defined by
\[\widetilde{\phi}(\widetilde{u}(\xi))=\exp(-\pi\|\xi\|^{2}).\]
Checking that $\widetilde{\phi}$ is well-defined reduces to showing that if we take $\xi_{1},\dots,\xi_{n}$,$\zeta_{1},\dots,\zeta_{m}\in \mathcal{H}$ and $\lambda_{1},\dots,\lambda_{n}$, $c_{1},\dots,c_{m}\in \CC$ with
\[\sum_{j=1}^{n}\lambda_{j}\widetilde{u}(\xi_{j})=\sum_{i=1}^{m}c_{j}\widetilde{u}(\zeta_{j}),\]
then
\[\sum_{j=1}^{n}\lambda_{j}\exp(-\pi\|\xi_{j}\|^{2})=\sum_{j=1}^{m}c_{j}\exp(-\pi\|\zeta_{j}\|^{2}).\]
 Replacing $\mathcal{H}$ with the span of $\{\xi_{1},\cdots,\xi_{n},\zeta_{1},\cdots,\zeta_{m}\}$ reduces the verification that $\widetilde{\phi}$ is well-defined to the case that $\mathcal{H}$ is finite-dimensional, where we have already shown that $\widetilde{\phi}$ is well-defined (by explicitly exhibiting it as integration with respect to the Gaussian measure as explained above). By similar reasoning, one checks that $\widetilde{\phi}(a^{*}a)\geq 0$ for $a\in \widetilde{\mathcal{A}}(\mathcal{H}).$ One now runs the GNS construction. That is, we let $\mathcal{K}$ be the completion of $\widetidle{\mathcal{A}}(\mathcal{H})$ under the inner product
\[\ip{a,b}=\widetilde{\phi}(b^{*}a).\]
We then have an induced homomorphism $\rho\colon\widetilde{\mathcal{A}}(\mathcal{H})\to B(\mathcal{K})$ given by
\[(\rho(a))(b)=ab\mbox{ for $b\in \widetilde{\mathcal{A}}(\mathcal{H})$}.\]
We then let $\mathcal{A}(\mathcal{H})=\overline{\rho(\widetilde{\mathcal{A}}(\mathcal{H}))}^{SOT}$, and  $\Omega=\widetilde{u}(0),u(\xi)=\rho(\widetilde{u}(\xi)).$ It is then straightforward to verify that $\mathcal{A}(\mathcal{H}),\Omega$ have the desired properties.

	By our remarks before this definition if $\rho,\Gamma$ are as in the definition then there is a standard probability space $(X_{\rho},\mu_{\rho})$ and a measure-preserving action $\Gamma\actson (X_{\rho},\mu_{\rho})$ so that
\[\Gamma\actson L^{\infty}(X_{\rho},\mu_{\rho})\cong \Gamma\actson (\mathcal{A}(\mathcal{H}),\phi).\]
Furthermore, this action is uniquely determined up to isomorphism.
 Note that for two orthogonal representations
\[\rho_{j}\colon\Gamma\to\mathcal{O}(\mathcal{H}_{j}),j=1,2\]
we have
\[\Gamma\actson (X_{\rho_{1}\oplus \rho_{2}},\mu_{\rho_{1}\oplus \rho_{2}})\cong (X_{\rho_{1}}\times X_{\rho_{2}},\mu_{\rho_{1}}\otimes \mu_{\rho_{2}}).\]
The definition via von Neumann algebras may be abstract, so let us mention a simple version of the definition in the case of a cyclic representation. Recall that if $t\in c_{c}(\Gamma,\RR)$ and $x\in \RR^{\Gamma}$ we use
\[t\cdot x=\sum_{g\in\Gamma}t(g)x(g).\]

\begin{proposition}\label{P:cyclic} Let $\Gamma$ be a countable discrete group and $\rho\colon \Gamma\to \mathcal{O}(\mathcal{H})$ an orthogonal representation. Suppose that there is a vector $\xi\in\mathcal{H}$ so that $\mathcal{H}=\overline{\Span\{\rho(g)\xi:g\in\Gamma\}}.$ Then the Gaussian action is isomorphic to the shift action on $\RR^{\Gamma}$ with measure determined by
\[\int_{\RR^{\Gamma}}\exp(2\pi i t\cdot x)\,d\mu_{\xi}(x)=\exp\left(-\pi\left\|\sum_{g\in\Gamma}t(g)\rho(g)\xi\right\|^{2}\right)\mbox{ for all $t\in c_{c}(\Gamma,\RR)$}.\]
\end{proposition}

\begin{proof}
Choose a realization $\Gamma\actson (\mathcal{A}(\mathcal{H}),\phi)\cong \Gamma\actson L^{\infty}(X_{\rho},\mu_{\rho}).$
By \cite{Taka} Proposition 5.3,  a sequence $a_{n}\in \mathcal{A}(\mathcal{H})$ converges in the strong operator topology to $a\in \mathcal{A}(\mathcal{H})$ if and only if
\[\phi((a-a_{n})^{*}(a-a_{n}))\to 0.\]
 Thus, $t\mapsto u(t\xi)$ is a strongly continuous one-parameter group in $\mathcal{A}(\mathcal{H}).$ Fix $\zeta\in \mathcal{H}.$  By Stone's Theorem there is a closed, densely-defined self-adjoint operator $\omega(\zeta)$ on $L^{2}(X_{\rho},\mu_{\rho})$
so that for all $s\in \RR$
\[\exp(2\pi i s\omega(\zeta))=u(s\zeta),\]
in the sense of functional calculus. Recall that $L^{\infty}(X_{\rho},\mu_{\rho})$ can be viewed as a von Neumann algebra (by multiplication operators). It is not hard to see that the elements $u\in L^{\infty}(X_{\rho},\mu_{\rho})$ whose multiplication operators are unitaries are almost everywhere equal to measurable functions
\[u\colon X\to \{z\in \CC:|z|=1\}.\]
 Thus we can identify $u(s\zeta)$ as a measurable function $X_{\rho}\to \{z\in \CC:|z|=1\}.$ For similar reasons, we can identify $\omega(\zeta)$ with a measurable function $X_{\rho}\to \RR.$ Thus for every $s\in\RR,$ it is true that for almost every $x\in X_{\rho}$ we have
\[\exp(2\pi i s\omega(\zeta)(x))=u(s\zeta)(x).\]
Define
\[\Phi\colon X_{\rho}\to \RR^{\Gamma}\]
by
\[\Phi(x)(g)=\omega(\xi)(g^{-1}x).\]
Then it is not hard to see that for every $s\in\RR$ and for all $g\in\Gamma$ we have
\[\exp(2\pi i s \Phi(x)(g))=u(s\rho(g)\xi)(x)\]
for almost every $x\in X_{\rho}.$ Define $\mu_{\xi}=\Phi_{*}\mu_{\rho}.$ Since the unitaries of the form
\[u(s\rho(g)\xi),s\in\RR,g\in\Gamma\]
generate all of $\mathcal{A}(\mathcal{H})$ we see that $\Phi$ gives an $\Gamma$-equivariant isomorphism
\[(\RR^{\Gamma},\mu_{\xi})\cong (X_{\rho},\mu_{\rho}).\]
Additionally for all $t\in c_{c}(\Gamma,\RR)$
\begin{align*}
\int_{\RR^{\Gamma}}\exp(2\pi i t\cdot x)\,d\mu_{\xi}(x)&=\int_{X}\exp\left(2\pi i \sum_{g\in\Gamma}t(g)\Phi(x)(g)\right)\,d\mu_{\rho}(g)\\
&=\phi\left(\omega\left(\sum_{g\in\Gamma}t(g)\rho(g)\xi\right)\right)\\
&=\exp\left(-\pi\left\|\sum_{g\in\Gamma}t(g)\rho(g)\right\|^{2}\right).
\end{align*}
\end{proof}

\subsection{Preliminaries on the Group von Neumann Algebra and Embedding Sequences}

For our purposes we will need to ``linearize'' a sofic approximation to approximations of algebras associated to $\Gamma.$ Let $\CC(\Gamma)$ be the ring of finite formal linear combinations of elements of $\Gamma$ with addition defined naturally and multiplication defined by
\[\left(\sum_{g\in\Gamma}a_{g}g\right)\left(\sum_{h\in\Gamma}b_{h}h\right)=\sum_{g\in\Gamma}\left(\sum_{h\in\Gamma}a_{h}b_{h^{-1}g}\right)g.\]
We will also define a conjugate-linear involution on $\CC(\Gamma)$ by
\[\left(\sum_{g\in\Gamma}a_{g}g\right)^{*}=\sum_{g\in\Gamma}\overline{a_{g^{-1}}}g.\]

Given a sofic approximation $\Sigma=(\sigma_{i}\colon\Gamma\to S_{d_{i}})$ and $\alpha=\sum_{g\in\Gamma}\alpha_{g}g\in \CC(\Gamma)$ we define $\sigma_{i}(\alpha)\in M_{d_{i}}(\CC)$ by
\[\sigma_{i}(\alpha)=\sum_{g\in\Gamma}\alpha_{g}\sigma_{i}(g).\]
In order to talk about the asymptotic properties of this extended sofic approximation, we will need a more analytic object associated to $\Gamma.$

Let $\lambda\colon\Gamma\to \mathcal{U}(\ell^{2}(\Gamma))$ be  the left regular representation defined by $(\lambda(g)\xi)(h)=\xi(g^{-1}h).$ We will continue to use $\lambda$ for the linear extension to $\CC(\Gamma)\to B(\ell^{2}(\Gamma)).$ The \emph{group von Neumann algebra} of $\Gamma$ is defined by
\[\overline{\lambda(\CC(\Gamma))}^{WOT}\]
where WOT denotes the weak operator topology. We will use $L(\Gamma)$ to denote the group von Neumann algebra. Define $\tau\colon L(\Gamma)\to \CC$ by
\[\tau(x)=\ip{x\delta_{e},\delta_{e}}.\]
We leave it as an exercise to the reader to verify that $\tau$ has the following properties.

\begin{list}{ \arabic{pcounter}:~}{\usecounter{pcounter}}
\item $\tau(1)=1,$\\
\item $\tau(x^{*}x)\geq 0,$  with equality if and only if $x=0$,\\
\item $\tau(xy)=\tau(yx),$  for all $x,y\in M,$
\item $\tau$ is weak operator topology continuous.
\end{list}
We call the third property the tracial property. We will typically view $\CC(\Gamma)$ as a subset of $L(\Gamma).$ In particular, we will use $\tau$ as well for the functional on $\CC(\Gamma)$ which is just the restriction of $\tau$ on $L(\Gamma).$

 In order to state our linearization of a sofic approximation properly, we give a general definition. By definition, $*$-algebra is  a complex algebra equipped with an involution $*$ which is conjugate linear and antimultiplicative.

\begin{definition}\emph{ A} tracial $*$-algebra \emph{ is a pair $(A,\tau)$ where $A$ is a $*$-algebra equipped with a linear functional $\tau\colon A\to \CC$ so that}
\begin{list}{ \arabic{pcounter}:~}{\usecounter{pcounter}}
\item $\tau(1)=1,$\\
\item $\tau(x^{*}x)\geq 0,$  \emph{with equality if and only if $x=0$},\\
\item $\tau(xy)=\tau(yx),$  \emph{for all $x,y\in M,$}
\item \emph{For all $a\in A,$ there is a $C_{a}>0$ so that for all $x\in A,$ $|\tau(x^{*}a^{*}ax)|\leq C_{a}\tau(x^{*}x).$}
\end{list}
\emph{ For $a,b\in A$ we let $\ip{a,b}=\tau(b^{*}a)$ and we let $\|a\|_{2}=\tau(a^{*}a)^{1/2}.$ Define $L^{2}(A,\tau)$ to be the Hilbert space completion of $A$ in this inner product. By condition $4$ of the definition, we have a representation $\lambda\colon A\to B(L^{2}(A,\tau))$ defined densely by $\lambda(a)x=ax$ for $x\in A.$ We let $\|a\|_{\infty}=\|\lambda(a)\|.$}
\end{definition}

We make $M_{n}(\CC)$ into a tracial $*$-algebra using $\tr=\frac{1}{n}\Tr$ where $\Tr$ is the usual trace. In particular, we use $\|A\|_{2}=\tr(A^{*}A)^{1/2}$ and $\|A\|_{\infty}$ will denote the operator norm.

We let $\CC[X_{1},\dots,X_{n}]$ be the free $*$-algebra on $n$-generators $X_{1},\dots,X_{n}.$  We will call elements of $\CC[X_{1},\dots,X_{n}]$ $*$-polynomials in $n$ indeterminates. For a $*$-algebra $A,$ for elements $a_{1},\dots,a_{n}\in A,$ there is a unique $*$-homomorphism $\CC[X_{1},\dots,X_{n}]\to A$ sending $X_{j}$ to $a_{j}$. If $P\in \CC[X_{1},\dots,X_{n}]$ we use $P(a_{1},\dots,a_{n})$ for the image of $P$ under this $*$-homomorphism.

\begin{definition}\emph{ Let $(A,\tau)$ be a tracial $*$-algebra. An} embedding sequence \emph{is a sequence $\Sigma=(\sigma_{i}\colon A\to M_{d_{i}}(\CC))$ such that}
\[\sup_{i}\|\sigma_{i}(a)\|_{\infty}<\infty,\mbox{ \emph{for all $a\in A$}}\]
\[\|P(\sigma_{i}(a_{1}),\dots,\sigma_{i}(a_{n}))-\sigma_{i}(P(a_{1},\dots,a_{n}))\|_{2}\to 0, \mbox{\emph{ for all $a_{1},\dots,a_{n}\in A$ and all $P\in \CC[X_{1},\dots,X_{n}]$}}\]
\[\tr(\sigma_{i}(a))\to \tau(a)\mbox{\emph{ for all $a\in A$}.}\]
\end{definition}

We will frequently use the following fact: if $x_{1},\dots,x_{n}\in A$ and $P\in \CC[X_{1},\dots,X_{n}]$ then
\begin{equation}\label{E:L2convergence}
\|P(\sigma_{i}(x_{1}),\dots,\sigma_{i}(x_{n}))\|_{2}\to \|P(x_{1},\dots,x_{n})\|_{2}.
\end{equation}
To see this, we notice that since
\[\|P(\sigma_{i}(x_{1}),\dots,\sigma_{i}(x_{n}))-\sigma_{i}(P(x_{1},\dots,x_{n}))\|_{2}\to 0,\]
it suffices to handle the case $n=1$  and $P(X)=X.$ Then,
\[\|\sigma_{i}(x)\|_{2}^{2}=\tr(\sigma_{i}(x)^{*}\sigma_{i}(x))\]
and since $\|\sigma_{i}(x)^{*}\sigma_{i}(x)-\sigma_{i}(x^{*}x)\|_{2}\to 0$ we have
\[|\tr(\sigma_{i}(x)^{*}\sigma_{i}(x))-\tr(\sigma_{i}(x^{*}x))|\to 0.\]
As
\[\tr(\sigma_{i}(x^{*}x))\to \|x\|_{2}^{2},\]
we have proved $(\ref{E:L2convergence}).$

The proof of the next two propositions will be left to the reader.
\begin{proposition}\label{P:soficextension} Let $\Gamma$ be a countable discrete sofic group with sofic approximation $\Sigma=(\sigma_{i}\colon \Gamma\to S_{d_{i}}).$ Extend $\Sigma$ to maps $\sigma_{i}:\CC(\Gamma)\to M_{d_{i}}(\CC)$ linearly. Then $\Sigma$ is an embedding sequence of $(\CC(\Gamma),\tau).$
\end{proposition}
\begin{proposition}\label{P:perturbation} Let $(A,\tau)$ be a tracial $*$-algebra and $\Sigma=(\sigma_{i}\colon A\to M_{d_{i}}(\CC))$ be an embedding sequence. If $\Sigma'=(\sigma_{i}'\colon A\to M_{d_{i}}(\CC))$ is another sequence of functions so that
\[\sup_{i}\|\sigma_{i}'(a)\|_{\infty}<\infty,\mbox{ for all $a\in A,$}\]
\[\|\sigma_{i}(a)-\sigma_{i}'(a)\|_{2}\to 0,\mbox{ for all $a\in A,$}\]
then $\Sigma'$ is an embedding sequence.
\end{proposition}

We will in fact need to extend our sofic approximation to the group von Neumann algebra. For this, we use the following.
\begin{lemma}[Lemma 5.5 in \cite{Me}] \label{L:soficextension}Let $\Gamma$ be a countable discrete group. Then any embedding sequence for $\CC(\Gamma)$ extends to one for $L(\Gamma).$
\end{lemma}
We will  use the preceding lemma when $\Gamma$ is sofic, in combination with Proposition \ref{P:soficextension}.

\subsection{Preliminaries on Real Subspaces of the Left Regular Representation}

We define the \emph{Fourier algebra} of $\Gamma$ as all functions $\phi\colon \Gamma\to \CC$ so that there is a linear functional $\Phi\colon L(\Gamma)\to\CC$ with $\Phi\big|_{\{x\in L(\Gamma):\|x\|_{\infty}\leq 1\}}$ being weak operator topology continuous so that
\[\Phi(\lambda(g))=\phi(g).\]
We will call $\Phi$ the continuous extension of $\phi$ (note that by continuity and Kaplansky's density theorem, $\Phi$ as above must be unique). We let $A(\Gamma)_{+}$ consist of all such $\phi$ where the continuous extension of $\phi$ is a positive linear functional. We let $\|\phi\|_{A(\Gamma)}=\|\Phi\|.$ By \cite{Taka} Theorem II.2.6 $\|\cdot\|_{A(\Gamma)}$ is a norm on $A(\Gamma)$ which makes $A(\Gamma)$ a Banach space. By \cite{Taka} Theorem V.3.15 we have that $A(\Gamma)$ consists of all functions of the form $g\mapsto \ip{\lambda(g)\xi,\eta}$ where $\xi,\eta\in\ell^{2}(\Gamma).$ Moreover
\[\|\phi\|=\inf\|\xi\|\|\eta\|\]
where infimum is over all $\xi,\eta$ so that $\phi(g)=\ip{\lambda(g)\xi,\eta}.$

For intuition, we leave it to the reader to verify that when $\Gamma$ is abelian, then $A(\Gamma)$ consists of all $\widehat{f}$ where $f\in L^{1}(\widehat{\Gamma}),$ and that $A(\Gamma)_{+}$ consists of all $\widehat{f}$ where $f\in L^{1}(\widehat{\Gamma})$ and $f\geq 0.$ We state a few basic (and well-known) properties of $A(\Gamma)$ in the following proposition. Lastly, if we are given $x\in L(\Gamma)$ we set
\[\widehat{x}=x\delta_{e}\in\ell^{2}(\Gamma).\]
\begin{proposition}\label{P:basicfacts} Let $\Gamma$ be a countable discrete group.
\begin{enumerate}
\item If $\rho\colon\Gamma\to \mathcal{U}(\mathcal{H})$ is a unitary representation with $\rho \ll \lambda_{\Gamma},$ and $\xi\in\mathcal{H},$ then $\phi(g)=\ip{\rho(g)\xi,\xi}$ is in $A(\Gamma)_{+}.$ In particular there is a $\zeta\in \ell^{2}(\Gamma)$ with $\ip{\rho(g)\xi,\xi}=\ip{\lambda(g)\zeta,\zeta}.$
\item If $\phi\in A(\Gamma),$ then $\overline{\phi}\in A(\Gamma)$ and $\|\overline{\phi}\|=\|\phi\|.$ If $\phi\in A(\Gamma)_{+},$ then so is $\overline{\phi}.$
\item For $x\in L(\Gamma)$ we define $\phi_{x}\colon \Gamma\to \CC$ by
\[\phi_{x}(g)=\tau(x\lambda(g)),\]
then $\phi_{x}\in A(\Gamma),$ and $\phi_{x}\in A(\Gamma)_{+}$ if and only if $x\in L(\Gamma)_{+}.$ Additionally
\[\|\phi_{x}\|=\tau(|x|),\]
and $\{\phi_{x}:x\in L(\Gamma)_{+}\}$ is dense in $A(\Gamma)_{+}.$
\item If $x,y\in L(\Gamma)$ and $\phi(g)=\ip{\lambda(g)\widehat{x},\widehat{x}},\psi(g)=\ip{\lambda(g)\widehat{y},\widehat{y}}$ then
\[\|\widehat{x}-\widehat{y}\|_{2}^{2}\leq \|\phi-\psi\|_{A(\Gamma)}.\]
\end{enumerate}
\end{proposition}

\begin{proof}

(1): The assumption $\rho \ll \lambda_{\Gamma}$ implies that we may extend $\rho$ to a $*$-homomorphism $\rho\colon L(\Gamma)\to B(\mathcal{H})$ so that
\[\rho\big|_{\{x\in L(\Gamma):\|x\|_{\infty}\leq 1\}}\]
is weak operator topology continuous. Thus the continuous extension $\Phi$ of $\phi$ is given by
\[\Phi(x)=\ip{\rho(x)\xi,\xi}.\]
The ``in particular'' part follows from the discussion preceding the proposition.

(2): Write
\[\phi(x)=\ip{\lambda(x)\xi,\eta},\]
with $\xi,\eta\in \ell^{2}(\Gamma).$ Then
\[\overline{\phi(x)}=\ip{\lambda(x)\overline{\xi},\overline{\eta}}\]
the conclusion follows easily from this equality.

(3): If $x\in L(\Gamma)_{+},$ then the continuous extension $\Phi$ of $\phi_{x}$ is given by
\[\Phi(y)=\tau(xy).\]
For $y\geq 0,$
\[\Phi(y)=\tau(xy)=\tau(y^{1/2}xy^{1/2})\geq 0\]
as $x\geq 0.$

	Conversely, if $\phi_{x}\in A(\Gamma)_{+}$ let $\Phi$ be the continuous extension of $\phi_{x}$ to $L(\Gamma).$  Then for all $y\in L(\Gamma)$
\[\ip{xy\delta_{e},y\delta_{e}}=\tau(y^{*}xy)=\tau(xyy^{*})=\Phi(yy^{*})\geq 0.\]
Since $\overline{L(\Gamma)\delta_{e}}=\ell^{2}(\Gamma),$ we see that $x\in L(\Gamma)_{+}.$ The norm equality and density statement are contained in \cite{TakesakiII} Lemma IX.2.12.

(4): This is the Powers-St\o rmer inequality (see \cite{BO} Proposition 6.2.4 for a proof which generalizes to our situation).

\end{proof}

In order to compute the entropy of Gaussian actions, we will need to discuss real subspaces of $\ell^{2}(\Gamma),$ and for this we need a real version of $L(\Gamma).$  We let $L_{\RR}(\Gamma)$ be all $x\in L(\Gamma)$ so that $\widehat{x}\in \ell^{2}(\Gamma,\RR).$  Recall that the convolution between $f\in c_{c}(\Gamma),g\in \CC^{\Gamma}$ is defined by
\[f*g(x)=\sum_{h\in\Gamma}f(h)g(h^{-1}x).\]
We similarly define $f*g$ if $f\in \CC^{\Gamma}$ and $g\in c_{c}(\Gamma).$
If $f\in c_{c}(\Gamma),$ we use
\[\check{f}=\sum_{g\in\Gamma}f(g)g\in \CC(\Gamma).\]
Note that for all $f\in c_{c}(\Gamma),x\in L(\Gamma)$ we have $xf=\widehat{x}*f.$ Thus
\[\|x\|_{\infty}=\sup_{f\in c_{c}(\Gamma),\|f\|_{2}=1}\|\widehat{x}*f\|_{2}.\]
Define $J\colon \ell^{2}(\Gamma)\to \ell^{2}(\Gamma)$ by $(Jf)(x)=\overline{f(x^{-1})},$ and observe that $J\widehat{x}=\widehat{x^{*}}.$ Thus for all $f\in c_{c}(\Gamma),$
\[\|f*\widehat{x}\|_{2}=\|J\widehat{x}*Jf\|_{2}=\|\widehat{x^{*}}*f\|_{2}\leq \|x^{*}\|_{\infty}\|f\|_{2}=\|x\|_{\infty}\|f\|_{2}.\]
Replacing $x$ with $x^{*}$ we see that
\[\|x\|_{\infty}=\sup_{f\in c_{c}(\Gamma),\|f\|_{2}=1}\|f*\widehat{x}\|_{2}.\]
Thus there is a unique bounded operator on $\ell^{2}(\Gamma)$ extending $f\mapsto f*\widehat{x}$ of norm equal to that of $x.$ We write the image of $\xi\in\ell^{2}(\Gamma)$ under this operator as $\xi x.$ We use
\[\ell^{2}(\Gamma)x=\{\xi x:\xi\in\ell^{2}(\Gamma)\}.\]

	By  Theorem 43.11 and Proposition 43.10 of \cite{ConwayOT}, we have
\[\{\widehat{x}:x\in L(\Gamma)\}=\left\{\xi\in\ell^{2}(\Gamma):\sup_{f\in c_{c}(\Gamma),\|f\|\leq 1}\|\xi*f\|_{2}<\infty\right\}.\]
For $\xi\in \ell^{2}(\Gamma),f\in c_{c}(\Gamma)$ we have
\[\|\xi*f\|_{2}=\|\overline{\xi}*\overline{f}\|_{2}\]
where bar denotes complex conjugation. It follows that there is a weak operator topology continuous, norm-one operator
\[C\colon L(\Gamma)\to L(\Gamma)\]
so that
\[\widehat{Cx}=\overline{\widehat{x}}.\]

\begin{lemma}\label{L:realembedding} Let $\Gamma$ be a countable discrete group and $\rho\colon\Gamma\to \mathcal{O}(\mathcal{H})$ an orthogonal representation on a real separable Hilbert space $\mathcal{H}$ with $\rho_{\CC}\ll\lambda_{\Gamma}.$ Suppose that $\rho$ has a cyclic vector. Then there is an orthogonal projection $p\in L_{\RR}(\Gamma)$ so that
\[\Gamma\actson \mathcal{H}\cong \Gamma\actson \ell^{2}(\Gamma,\RR)p.\]

\end{lemma}

\begin{proof}

Let $\xi$ be a cyclic vector for $\mathcal{H},$ and define $\phi\colon\Gamma\to \RR$ by
\[\phi(g)=\ip{\rho(g)\xi,\xi},\]
then $\phi(g)\in A(\Gamma)_{+}.$ By (3) of the preceding Proposition, we may find $x_{n}\in L(\Gamma)_{+}$ so that
\[\|\phi-\phi_{x_{n}}\|\to 0.\]
Letting $y_{n}=\frac{x_{n}+Cx_{n}}{2}$ and using that $\overline{\phi}=\phi$ we find that
\[\|\phi-\phi_{y_{n}}\|\to 0.\]
Let $\zeta_{n}=\widehat{y_{n}^{1/2}}.$ Approximating the square root function by polynomials, we see that $\zeta_{n}\in \ell^{2}(\Gamma,\RR).$ Note that
\[\phi_{y_{n}}(g)=\tau(\lambda(g)y_{n}^{1/2}y_{n}^{1/2})=\tau(y_{n}^{1/2}\lambda(g)y_{n}^{1/2})=\ip{\lambda(g)\zeta_{n},\zeta_{n}}.\]
 By the Powers-St\o rmer inequality (i.e. (4) of the preceding Proposition),
\[\|\zeta_{n}-\zeta_{m}\|\leq \|\phi_{y_{n}}-\phi_{y_{m}}\|^{1/2}\]
so that $\zeta_{n}$ is a Cauchy sequence. Hence $\zeta_{n}$ converges to a $\zeta\in\ell^{2}(\Gamma,\RR).$ Also,
\[\phi(g)=\lim_{n\to \infty}\phi_{y_{n}}(g)=\lim_{n\to \infty}\ip{\lambda(g)\zeta_{n},\zeta_{n}}=\lim_{n\to\infty}\ip{\lambda(g)\zeta,\zeta}.\]
Thus
\[\Gamma\actson \mathcal{H}\cong \Gamma\actson \overline{\Span\{\lambda(g)\zeta:g\in\Gamma\}}.\]
Let $P$ be the projection from $\ell^{2}(\Gamma,\RR)$ onto $\overline{\Span\{\lambda(g)\zeta:g\in\Gamma\}}$ and let $P_{\CC}$ denote its complexification as an operator on $\ell^{2}(\Gamma).$  As $P_{\CC}$ commutes with $\lambda(g),$  Theorem 43.11 of \cite{ConwayOT} shows that  there is a unique orthogonal projection $p\in L(\Gamma)$ so that  for all $\xi\in \ell^{2}(\Gamma)$ we have $P_{\CC}\xi=\xi p.$ Moreover,
\[\widehat{p}(g)=\ip{\widehat{p},\delta_{g}}=\ip{P(\delta_{e}),\delta_{g}}\in \RR\]
so $p\in L_{\RR}(\Gamma).$ Thus
\[\overline{\left\{\sum_{g\in\Gamma} f(g)\lambda(g)\zeta:f\in c_{c}(\Gamma,\RR)\right\}}=P(\ell^{2}(\Gamma,\RR))=\ell^{2}(\Gamma,\RR)p.\]

\end{proof}

We will need to extend a sofic approximation to an embedding sequence of $L(\Gamma)$ as in Lemma 5.5 in \cite{Me}, however we will also want $\sigma_{i}(L_{\RR}(\Gamma))\subseteq M_{d_{i}}(\RR).$

\begin{proposition}\label{P:realextension} Let $\Gamma$ be a countable discrete group and $\Sigma=(\sigma_{i}\colon \Gamma\to S_{d_{i}})$ a sofic approximation.

(i): There exists an embedding sequence $\Sigma=(\sigma_{i}\colon L(\Gamma)\to M_{d_{i}}(\CC))$ such that
\[\sigma_{i}(\alpha)=\sum_{g\in\Gamma}\widehat{\alpha}(g)\sigma_{i}(g)\]
and $\sigma_{i}(L_{\RR}(\Gamma))\subseteq M_{d_{i}}(\RR).$

(ii) If $\Sigma=(\sigma_{i}\colon L(\Gamma)\to M_{d_{i}}(\CC))$ is as in $(i),$ and $p\in L_{\RR}(\Gamma)$  is an orthogonal projection, then there are orthogonal projections $p_{i}\in M_{d_{i}}(\RR)$ so that
\[\|\sigma_{i}(p)-p_{i}\|_{2}\to 0.\]

\end{proposition}

\begin{proof}

(i): By Lemma \ref{L:soficextension} we may extend to some embedding sequence $\Sigma=(\sigma_{i}\colon L(\Gamma)\to M_{d_{i}}(\CC))$ so that
\[\sigma_{i}(\alpha)=\sum_{g\in\Gamma}\widehat{\alpha}(g)\sigma_{i}(g).\]
By Proposition \ref{P:perturbation}, it suffices to show that for all $x\in L_{\RR}(\Gamma)\setminus \RR(\Gamma),$ there are $x_{i}\in M_{d_{i}}(\RR)$ so that
\[\sup_{i}\|x_{i}\|_{\infty}<\infty\]
and
\[\|x_{i}-\sigma_{i}(x)\|_{2}\to 0.\]
 For $A\in M_{n}(\CC)$ define $\overline{A}\in M_{n}(\CC)$ by $\overline{A}_{ij}=\overline{A_{ij}},$ then
\[\|\overline{A}\|_{\infty}=\|A\|_{\infty}\]
\[\|\overline{A}\|_{2}=\|A\|_{2}.\]
It suffices to show that
\[\|\sigma_{i}(Cx)-\overline{\sigma_{i}(x)}\|_{2}\to 0\]
for all $x\in L(\Gamma).$ Indeed, assuming we have the above convergence we may then redefine $\sigma_{i}(x)$ for $x\in L_{\RR}(\Gamma)$ by $\frac{1}{2}(\sigma_{i}(x)+\overline{\sigma_{i}(x)}).$ So let $x\in L(\Gamma),$ let $\varepsilon>0,$ and let $\alpha\in \CC(\Gamma)$ be such that
\[\|x\delta_{e}-\lambda(\alpha)\delta_{e}\|_{2}<\varepsilon.\]
Since $\sigma_{i}(C\lambda(\alpha))=\overline{\sigma_{i}(\alpha)},$
\begin{align*}
\|\sigma_{i}(Cx)-\overline{\sigma_{i}(x)}\|_{2}&\leq \|\sigma_{i}(Cx)-\sigma_{i}(C\lambda(\alpha))\|_{2}+\|\overline{\sigma_{i}(x)}-\overline{\sigma_{i}(\lambda(\alpha))}\|_{2}\\
&= \|\sigma_{i}(Cx)-\sigma_{i}(C\lambda(\alpha))\|_{2}+\|\sigma_{i}(x)-\sigma_{i}(\lambda(\alpha))\|_{2}.
\end{align*}
Letting $i\to \infty$ and using that $\sigma_{i}$ is a sofic approximation we find that
\[\limsup_{i\to \infty}\|\sigma_{i}(Cx)-\overline{\sigma_{i}(x)}\|_{2}\leq \|Cx\delta_{e}-C\lambda(\alpha)\delta_{e}\|+\|x\delta_{e}-\lambda(\alpha)\delta_{e}\|_{2}<2\varepsilon.\]
Letting $\varepsilon\to 0$ proves $(i).$

(ii): Since $\sigma_{i}$ is an embedding sequence we have
\[\||\sigma_{i}(p)|^{2}-\sigma_{i}(p)\|_{2}\to 0,\]
\[\||\sigma_{i}(p)|^{2}-|\sigma_{i}(p)|^{4}\|_{2}\to 0.\]
By functional calculus,
\begin{align*}
\|\chi_{[1/2,3/2]}(|\sigma_{i}(p)|^{2})-|\sigma_{i}(p)|^{2})\|_{2}^{2}&=\|\chi_{\{t:|t-1|>1/2\}}(|\sigma_{i}(p)|^{2})\|_{2}^{2}\\
&=\tr(\chi_{\{t:|t-1|>1/2\}}(|\sigma_{i}(p)|^{2}))\\
&\leq16\tr(||\sigma_{i}(p)|^{2}-|\sigma_{i}(p)|^{4}|^{2})\\
&=16\||\sigma_{i}(p)|^{2}-|\sigma_{i}(p)|^{4}\|_{2}^{2}\\
&\to 0.
\end{align*}
Thus setting $p_{i}=\chi_{[1/2,3/2]}(|\sigma_{i}(p)|^{2})$ completes the proof.

\end{proof}

Lastly, we will need an analogous definition of singularity, as in the unitary case. If $\rho_{j}\colon \Gamma\to \mathcal{O}(\mathcal{H}_{j}),j=1,2$ are two orthogonal representations, we use $\Hom_{\Gamma}(\rho_{1},\rho_{2})$ for the space of \emph{real}, linear, bounded, $\Gamma$-equivariant maps from $\mathcal{H}_{1}\to \mathcal{H}_{2}.$  We say that $\rho_{1},\rho_{2}$ are mutually singular, written $\rho_{1}\perp \rho_{2},$  if whenever $\rho_{j}'$ is a nonzero subrepresentation of $\rho_{j}$ for $j=1,2$ we have that $\rho_{1}',\rho_{2}'$ are not isomorphic. Similarly, we say that $\rho_{1}\ll \rho_{2}$ if $\rho_{1}$ is embeddable into $\rho_{2}^{\oplus \infty}.$

\begin{lemma}\label{L:realsingularity} Let $\Gamma$ be a countable discrete group, and $\rho_{j}\colon\Gamma\to \mathcal{O}(\mathcal{H}_{j})$ two orthogonal representations. The following are equivalent.

(i): $\rho_{1}\perp\rho_{2}$

(ii): $\Hom_{\Gamma}(\rho_{1},\rho_{2})=\{0\},$

(iii): $\Hom_{\Gamma}(\rho_{2},\rho_{1})=\{0\}.$

\end{lemma}

\begin{proof} The proof of $(ii)$ equivalent to $(iii)$ and $(ii)$ implies $(i)$ is the same as in Proposition 4.2 of \cite{Me6}. We can copy the proof of $(i)$ implies $(ii)$ in Proposition 4.2 of \cite{Me6}  provided we prove an analogue of the Polar decomposition.

So let $T\colon \mathcal{H}_{1}\to \mathcal{H}_{2}$ be a bounded operator. Let
\[T_{\CC}\colon \mathcal{H}_{1,\CC}\to \mathcal{H}_{2,\CC}\]
be the complexification of $T.$ Using $T^{t}$ for the real Banach space transpose of $T,$
\[T_{\CC}^{*}\big|_{\mathcal{H}_{2}}=T^{t}.\]
Thus
\[T_{\CC}^{*}T_{\CC}(\mathcal{H}_{1})\subseteq \mathcal{H}_{1},\]
approximating the square root function by polynomials, we see that
\[|T_{\CC}|(\mathcal{H}_{1})\subseteq \mathcal{H}_{1}.\]
Let
\[T_{\CC}=U|T_{\CC}|\]
be the polar decomposition. As
\[U=SOT-\lim_{\varepsilon\to 0}T_{\CC}(|T_{\CC}|+\varepsilon)^{-1},\]
we find that $U(\mathcal{H}_{1})\subseteq \mathcal{H}_{2}.$ The rest is as in Proposition 4.2 of \cite{Me6}.

\end{proof}

We now show that in the case of the left regular representation, that the concepts of singularity and absolute continuity in the real case are related to the complex case.

\begin{lemma}\label{L:realproblemsman} Let $\Gamma$ be a countable discrete group. Let $\rho\colon \Gamma\to \mathcal{O}(\mathcal{H})$ be an orthogonal representation. Then $\rho_{\CC}\ll \lambda_{\Gamma}$ (in the unitary sense) if and only if $\rho\ll \lambda_{\Gamma,\RR}.$ Similarly $\rho_{\CC}\perp \lambda_{\Gamma}$ if and only if $\rho \perp \lambda_{\Gamma,\RR}.$
\end{lemma}

\begin{proof}

Suppose $\rho_{\CC}\ll \lambda_{\Gamma}.$ Applying Zorn's Lemma to write $\rho$ as a direct sum of cyclic representations and applying Lemma \ref{L:realembedding} we see that $\rho\ll \lambda_{\Gamma,\RR}.$ The converse is even easier.

Suppose $\rho_{\CC}\perp\lambda_{\Gamma}.$ If $\mathcal{K}\subseteq \mathcal{H}$ is a closed, linear, $\Gamma$-invariant subspace so that $\rho\big|_{\mathcal{K}}$ embeds into $\lambda_{\Gamma,\RR},$ then by complexification we have that $\rho_{\CC}\big|_{\mathcal{K_{\CC}}}$ embeds into $\lambda_{\Gamma}.$ Conversely, suppose $T\in \Hom_{\Gamma}(\rho_{\CC},\lambda_{\Gamma}).$ Then $T\big|_{\mathcal{H}}$ is a $\Gamma$-equivariant, bounded, real-linear map
\[\mathcal{H}\to \ell^{2}(\Gamma).\]
As $\ell^{2}(\Gamma)\cong \ell^{2}(\Gamma,\RR)^{\oplus 2}$ as a real representation we see that $T\big|_{\mathcal{H}}=0,$ by the previous Lemma. Since $\mathcal{H}$ spans $\mathcal{H}_{\CC}$ (as a complex vector space) we see that  $T=0.$

\end{proof}

The following is proved in the same way as Proposition 4.3 of \cite{Me6}.

\begin{proposition}\label{P:LebesgueDecompositionOR} Let $\Gamma$ be a countable discrete group and $\rho_{j}\colon\Gamma\to \mathcal{O}(\mathcal{H}_{j}),j=1,2$ be two orthogonal representations. Then
\[\rho_{1}=\rho_{1,s}\oplus \rho_{1,c}\]
where $\rho_{1,s}\perp \rho_{2},$ and $\rho_{1,c}\ll\rho_{2}.$

\end{proposition}

\subsection{Sofic Entropy of Gaussian Actions}\label{S:Gauss2}

In this section we compute  the entropy of Gaussian actions. Let us first start with a very simple corollary of Theorem 4.4 of \cite{Me6}.

\begin{cor}\label{C:zerogaussian}

Let $\Gamma$ be a countable discrete sofic group with sofic approximation $\Sigma.$ Let $\rho\colon\Gamma\to\mathcal{O}(\mathcal{H})$ be an orthogonal representation on a separable Hilbert space $\mathcal{H}.$ Suppose that $\rho\perp \lambda_{\Gamma,\RR}.$ Then if $\Gamma\actson (X_{\rho},\mu_{\rho})$ is the corresponding Gaussian action we have
\[h_{\Sigma,\mu_{\rho}}(X_{\rho},\Gamma)\leq 0.\]

\end{cor}

\begin{proof} As in proposition \ref{P:cyclic}, for any $\xi\in\mathcal{H}$ we can find a unique
\[\omega(\xi)\colon X_{\rho}\to \RR\]
so that for all $t\in \RR$ we have
\[\exp(2\pi i t \omega(\xi)(x))=u(t\xi)(x)\]
for almost every $x\in X_{\rho}.$  By uniqueness, we have that
\[\omega(\xi+\eta)=\omega(\xi)+\omega(\eta)\]
almost everywhere. By \cite{PetersonSinclair} we have that $\omega(\xi)\in L^{2}(X_{\rho},\mu_{\rho})$ for all $\xi\in \mathcal{H}$ and that in fact
\[\Gamma\actson \overline\Span\{\omega(\xi):\xi\in \mathcal{H}\}\cong \Gamma\actson\mathcal{H}\perp \lambda_{\Gamma}.\]
We have that $L^{\infty}(X_{\rho},\mu_{\rho})$ is generated as a von Neumann algebra by
\[u(t\rho(g)\xi)=\exp(2\pi it \omega(\rho(g)\xi)) \mbox{ for $t\in \RR,\xi\in\mathcal{H},g\in\Gamma$}.\]
From this, it is not hard to argue that the sigma-algebra generated by
\[\{\omega(\xi)^{-1}(E):E\subseteq \CC\mbox{ is Borel},\xi\in\mathcal{H}\}\]
is all measurable sets (up to measure zero).  Thus by Theorem 4.4 of \cite{Me6} and Lemma \ref{L:realproblemsman} we know that
\[h_{\Sigma,\mu_{\rho}}(X_{\rho},\Gamma)\leq 0.\]

\end{proof}

We turn to the computation of sofic entropy of Gaussian actions in the case that $\rho_{\CC}\ll\lambda_{\Gamma}.$ We will need the following general Lemmas.

\begin{lemma}\label{L:generalmicrostates} Let $\Gamma$ be  a countable discrete sofic group with sofic approximation $\Sigma=(\sigma_{i}\colon\Gamma\to S_{d_{i}}).$ Let $X$ be a Polish space with a bounded, compatible (not necessarily complete) metric $\Delta',$ and fix $x_{0}\in X.$ Let $\Gamma\actson X^{\Gamma}$ be the Bernoulli action, and give  $X^{\Gamma}$ the dynamically generating pseudometric
\[\Delta(x,y)=\Delta'(x(e),y(e)).\]
For a finite $E\subseteq\Gamma$ and $x\in X^{d_{i}}$ define $\phi^{(E)}_{x}\colon\{1,\dots,d_{i}\}\to X^{\Gamma}$ by
\[\phi^{(E)}_{x}(j)(g)=\begin{cases}
x(\sigma_{i}(g)^{-1}(j)),& \textnormal{ if $g\in E$}\\
x_{0},& \textnormal{ if $g\notin E$.}
\end{cases}\]
For any finite $F\subseteq\Gamma,$ there is a finite $E_{0}\subseteq\Gamma$ so that if $E\subseteq \Gamma$ is a finite set containing $E_{0},$ then for all $\delta>0$ and for all large $i$ we have
\[\phi^{(E)}_{x}\in\Map(\Delta,F,\delta,\sigma_{i}).\]
\end{lemma}
\begin{proof}
For $h\in F,$ we have
\[\Delta_{2}(h\phi^{(E)}_{x},\phi^{(E)}_{x}\circ\sigma_{i}(h))^{2}=\frac{1}{d_{i}}\sum_{j=1}^{d_{i}}\Delta'(\phi^{(E)}_{x}(j)(h^{-1}),\phi^{(E)}_{x}(\sigma_{i}(h)(j))(e))^{2}.\]
If $E\supseteq F\cup F^{-1}\cup\{e\},$ then we have
\[\Delta_{2}(h\phi^{(E)}_{x},\phi^{(E)}_{x}\circ\sigma_{i}(h))^{2}=\frac{1}{d_{i}}\sum_{j:\sigma_{i}(h^{-1})^{-1}(j)\ne \sigma_{i}(h)(j)}\Delta'(x(\sigma_{i}(h^{-1})^{-1}(j)),x(\sigma_{i}(h)(j)))^{2}.\]
Let $M$ be the diameter of $(X,\Delta')$ then we have
\[\Delta_{2}(h\phi^{(E)}_{x}\circ \sigma_{i}(h),h\phi^{(E)}_{x})^{2}\leq Mu_{d_{i}}(\{1\leq j\leq d_{i}:\sigma_{i}(h^{-1})^{-1}(j)\ne \sigma_{i}(h)(j)\})\to 0\]
by soficity. Setting $E_{0}=F\cup F^{-1}\cup\{e\}$ completes the proof.

\end{proof}

\begin{lemma} Let $\Gamma$ be a countable discrete sofic group with sofic approximation $\Sigma=(\sigma_{i}\colon\Gamma\to S_{d_{i}}).$ By Lemma \ref{L:soficextension} extend $\Sigma$ to an approximation sequence $\Sigma=(\sigma_{i}\colon L(\Gamma)\to M_{d_{i}}(\CC))$ so that for $\alpha\in\CC(\Gamma)$
\[\sigma_{i}(\alpha)=\sum_{g\in\Gamma}\widehat{\alpha}(g)\sigma_{i}(g).\]
Fix a finite $F\subseteq\Gamma$ and a  $x\in L(\Gamma),$ and let $x_{i}\in M_{d_{i}}(\CC)$ be such that
\[\sup_{i}\|x_{i}\|_{\infty}<\infty,\]
\[\|\sigma_{i}(x)-x_{i}\|_{2}\to 0.\]
Then the following statements hold.

(i): There is a sequence $C_{i}\subseteq\{1,\dots,d_{i}\}$ so that
\[u_{d_{i}}(C_{i})\to 1\]
with
\[\lim_{i\to \infty}\sup_{\substack{t\in \RR^{F},s\in\RR^{F},\\ j\in C_{i}}}\frac{\left|\|\sigma_{i}(\check{t})x_{i}\sigma_{i}(\check{s})e_{j}\|_{\ell^{2}(d_{i})}^{2}-\|\lambda(\check{t})x\lambda(\check{s})\delta_{e}\|_{2}^{2}\right|}{\|t\|_{\ell^{2}(F)}^{2}\|s\|_{\ell^{2}(F)}^{2}}=0.\]

(ii): There is a sequence $A_{i}\subseteq \{1,\dots,d_{i}\}^{2}$ so that
\[u_{d_{i}}\otimes u_{d_{i}}(A_{i})\to 1,\]
and
\[\lim_{i\to \infty}\sup_{\substack{t\in \RR^{F},s\in\RR^{F},\\ (j,k)\in A_{i}}}\frac{\left|\|\sigma_{i}(\check{t})^{*}x_{i}e_{j}-\sigma_{i}(\check{s})^{*}x_{i}e_{k}\|_{\ell^{2}(d_{i})}^{2}-\|\lambda(\check{t})^{*}x\delta_{e}\|_{2}^{2}-\|\lambda(\check{s})x\delta_{e}\|_{2}^{2}\right|}{(\|t\|_{\ell^{2}(F)}+\|s\|_{\ell^{2}(F)})^{2}}=0.\]

\end{lemma}

\begin{proof}

(i): We first handle the case that $x\in \CC(\Gamma)$ and $x_{i}=\sigma_{i}(x).$ We may choose a $C_{i}$ so that
\[u_{d_{i}}(C_{i})\to 1,\]
\[\sigma_{i}(\check{t}x\check{s})e_{j}=\sigma_{i}(\check{t})\sigma_{i}(x)\sigma_{i}(\check{s})e_{j},\]
 for $j\in C_{i}.$ By soficity, we may also force that
\[\sigma_{i}(g)e_{j}\ne \sigma_{i}(h)e_{j}\mbox{ for $g\ne h$ in $\supp(\check{t}x\check{s}),$ and $j\in C_{i},$}\]
while still having
\[u_{d_{i}}(C_{i})\to 1.\]
In this case it is  easy  to see that
\[\|\sigma_{i}(\check{t})\sigma_{i}(x)\sigma_{i}(\check{s})e_{j}\|_{\ell^{2}(d_{i})}^{2}=\|\sigma_{i}(\check{t}x\check{s})e_{j}\|_{\ell^{2}(d_{i})}^{2}=\|\lambda(\check{t})x\lambda(\check{s})\delta_{e}\|_{\ell^{2}(d_{i})}^{2}\]
for all $j\in C_{i}.$
	
	Now we handle the general case. Let $\varepsilon>0,$ and choose $\alpha\in \CC(\Gamma)$ so that
\[\|(\lambda(\alpha)-x)\delta_{e}\|_{2}<\varepsilon.\]
Since $\sigma_{i}$ is an approximation sequence by (\ref{E:L2convergence}) we have for all $g,h\in\Gamma$ that
\[\lim_{i\to\infty}\|\sigma_{i}(g)\sigma_{i}(\alpha)\sigma_{i}(h)-\sigma_{i}(g)x_{i}\sigma_{i}(h)\|_{2}=\|\lambda(g)(\lambda(\alpha)-x)\lambda(h)\delta_{e}\|_{2}<\varepsilon.\]
Thus for all large $i$ we have
\begin{align*}
\frac{1}{d_{i}}\sum_{g,h\in F}\sum_{j=1}^{d_{i}}\|\sigma_{i}(g)\sigma_{i}(\alpha)\sigma_{i}(h)e_{j}-\sigma_{i}(g)x_{i}\sigma_{i}(h)e_{j}\|_{\ell^{2}(d_{i})}^{2}&=\sum_{g,h\in F}\|\sigma_{i}(g)\sigma_{i}(\alpha)\sigma_{i}(h)-\sigma_{i}(g)\sigma_{i}(\alpha)\sigma_{i}(h)\|_{L^{2}(M_{d_{i}}(\CC),\tr)}^{2}\\
&<\varepsilon|F|^{2}.
\end{align*}
For such $i,$ we may find a $C_{i}\subseteq \{1,\dots,d_{i}\}$ with
\[u_{d_{i}}(C_{i})\geq (1-\sqrt{\varepsilon}|F|^{2})\]
so that
\[\|\sigma_{i}(g)\sigma_{i}(\alpha)\sigma_{i}(h)e_{j}-\sigma_{i}(g)x_{i}\sigma_{i}(h)e_{j}\|_{\ell^{2}(d_{i})}^{2}<\sqrt{\varepsilon}\]
for all $g,h\in F$ and all large $i,$ and $j\in C_{i}.$ Thus for all $j\in C_{i},$ and all $s,t\in\RR^{F}$
\begin{align*}
\|\sigma_{i}(\check{t})x_{i}\sigma_{i}(\check{s})e_{j}-\sigma_{i}(\check{t})\sigma_{i}(\alpha)\sigma_{i}(\check{s})e_{j}\|_{\ell^{2}(d_{i})}&\leq \sum_{g,h\in\Gamma}|t(g)s(h)|\|\sigma_{i}(g)x_{i}\sigma_{i}(h)e_{j}-\sigma_{i}(g)\sigma_{i}(\alpha)\sigma_{i}(h)e_{j}\|_{\ell^{2}(d_{i})}\\
&\leq \varepsilon^{1/4}|F|^{2}\|t\|_{\ell^{2}(F)}\|s\|_{\ell^{2}(F)}.
\end{align*}
Therefore for $j\in C_{i}$ and all large $i,$ and all $t,s\in\RR^{F}$
\[|\|\sigma_{i}(\check{t})x_{i}\sigma_{i}(\check{s})e_{j}\|_{\ell^{2}(d_{i})}-\|\sigma_{i}(\check{t})\sigma_{i}(\alpha)\sigma_{i}(\check{s})e_{j}\|_{2}|\leq \varepsilon^{1/4}|F|^{2}\|t\|_{\ell^{2}(F)}\|s\|_{\ell^{2}(F)}.\]
By the first part, we can find a $C_{i}'\subseteq\{1,\dots,d_{i}\}$ with
\[u_{d_{i}}(C_{i}')\to 1,\]
 and
\[\|\sigma_{i}(\check{t})\sigma_{i}(\alpha)\sigma_{i}(\check{s})e_{j}\|_{\ell^{2}(d_{i})}^{2}=\|\lambda(\check{t})\lambda(\alpha)\lambda(\check{s})\delta_{e}\|_{2}^{2},\]
for all $j\in C_{i}',$ and $t,s\in \RR^{F}.$ Thus for all large $i$ and all $t,s\in\RR^{F},$ $j\in C_{i}\cap C_{i}',$
\begin{align*}
|\|\sigma_{i}(\check{t})x_{i}\sigma_{i}(\check{s})e_{j}\|_{\ell^{2}(d_{i})}-\|\lambda(\check{t})x\lambda(\check{s})\delta_{e}\|_{\ell^{2}(\Gamma)}|&\leq \varepsilon^{1/4}|F|^{2}\|t\|_{\ell^{2}(F)}\|s\|_{\ell^{2}(F)}+\varepsilon\|\lambda(\check{t})\|_{\infty}\|\lambda(\check{s})\|_{\infty}\\
&\leq  \varepsilon^{1/4}|F|^{2}\|t\|_{\ell^{2}(F)}\|s\|_{\ell^{2}(F)}+\varepsilon\|t\|_{\ell^{1}(F)}\|s\|_{\ell^{1}(F)}\\
&\leq \varepsilon^{1/4}|F|\|t\|_{\ell^{2}(F)}\|s\|_{\ell^{2}(F)}+|F|\varepsilon\|t\|_{\ell^{2}(F)}\|s\|_{\ell^{2}(F)}
\end{align*}
Since
\[u_{d_{i}}(C_{i}\cap C_{i}')\geq (1-2\sqrt{\varepsilon})|F|\]
for all large $i,$ we can use a diagonal argument to complete the proof of $(i).$

(ii): Let $C_{i}$ be as $(i).$ Note that
\[\|\sigma_{i}(\check{t})^{*}x_{i}e_{j}-\sigma_{i}(\check{s})^{*}x_{i}e_{k}\|_{\ell^{2}(d_{i})}^{2}=\|\sigma_{i}(\check{t})^{*}x_{i}e_{j}\|_{\ell^{2}(d_{i})}^{2}+\|\sigma_{i}(\check{s})^{*}x_{i}e_{k}\|_{\ell^{2}(d_{i})}^{2}-2\Rea(\ip{x_{i}^{*}\sigma_{i}(\check{s})\sigma_{i}(\check{t})^{*}x_{i}e_{j},e_{k}}_{\ell^{2}(d_{i})}).\]
We look for $A_{i}$ with  $A_{i}\subseteq C_{i}\times C_{i}.$ By $(i)$ it is enough to find an $A_{i}\subseteq C_{i}\times C_{i}$ with
\[\lim_{i\to\infty}\sup_{s,t\in\RR^{F},(j,k)\in A_{i}}\frac{|\ip{x_{i}^{*}\sigma_{i}(\check{s})\sigma_{i}(\check{t})^{*}x_{i}e_{j},e_{k}}_{\ell^{2}(d_{i})}|}{(\|t\|_{\ell^{2}(F)}+\|s\|_{\ell^{2}(F)})^{2}}=0,\]
\[u_{d_{i}}\otimes u_{d_{i}}(A_{i})\to 1.\]
We prove that
\begin{equation}\label{E:section6dalhgladjg}
\lim_{i\to\infty}\frac{1}{d_{i}^{2}(\|t\|_{\ell^{2}(F)}+\|s\|_{\ell^{2}(F)})^{2}}\sum_{1\leq j,k\leq d_{i}}|\ip{x_{i}^{*}\sigma_{i}(\check{s})\sigma_{i}(\check{t})^{*}x_{i}e_{j},e_{k}}_{\ell^{2}(d_{i})}|=0,
\end{equation}
which will clearly prove the existence of such an $A_{i}.$ Again, as in $(i)$ we first do this for $x_{i}=\sigma_{i}(\alpha)$ for $\alpha\in\CC(\Gamma).$  Let $D_{i}$ be the set of all $j$ so that
\[\sigma_{i}(\alpha)^{*}\sigma_{i}(\check{s})\sigma_{i}(\check{t})^{*}\sigma_{i}(\alpha)e_{j}=\sigma_{i}(\alpha^{*}st^{*}\alpha)e_{j}.\]
We use $o(1)$ for any expression which tends to $0$ as $i\to\infty.$ Then
\[\lim_{i\to \infty} u_{d_{i}}(D_{i})=1,\]
and
\begin{align*}
\frac{1}{d_{i}^{2}}\sum_{1\leq j,k\leq d_{i}}|\ip{\sigma_{i}(\alpha)^{*}\sigma_{i}(\check{s})\sigma_{i}(\check{t})^{*}&\sigma_{i}(\alpha)e_{j},e_{k}}_{\ell^{2}(d_{i})}|^{2}\leq o(1)\\
&+\frac{1}{d_{i}^{2}}\sum_{1\leq j,k\leq d_{i},j\in D_{i}}\sum_{g,h\in \Gamma}(\alpha^{*}\check{s}\check{t}^{*}\alpha)(g)\overline{(\alpha^{*}\check{s}\check{t}^{*}\alpha)(h)}\ip{\sigma_{i}(g)e_{j},e_{k}}_{\ell^{2}(d_{i})}\ip{e_{k},\sigma_{i}(h)e_{j}}_{\ell^{2}(d_{i})}\\
&=o(1)+\frac{1}{d_{i}^{2}}\sum_{j\in D_{i}}\sum_{g\in\Gamma}|(\alpha^{*}\check{s}\check{t}^{*}\alpha)(g)|^{2}\\
&=o(1)+\frac{1}{d_{i}}u_{d_{i}}(D_{i})\|\alpha^{*}\check{s}\check{t}^{*}\alpha\|_{2}^{2}\\
&\leq o(1)+\frac{1}{d_{i}}u_{d_{i}}(D_{i})\|\alpha\|_{\infty}\|s\|_{\ell^{1}(F)}^{2}\|t\|_{\ell^{1}(F)}^{2}\|\alpha\|_{\infty}\\
&\leq o(1)+\frac{1}{d_{i}}|F|^{2}\|t\|_{\ell^{2}(F)}^{2}\|s\|_{\ell^{2}(F)}^{2}\|\alpha\|_{\infty}^{2}(1+o(1))\\
&\to 0.
\end{align*}
This proves (\ref{E:section6dalhgladjg}) by the Cauchy-Schwartz inequality. The general case follows by approximation as in (i).

\end{proof}
For notation, if $E\subseteq\Gamma$ is finite, and $f\colon \RR^{E}\to\CC$ is measurable, we let $f\otimes 1_{\RR^{\Gamma\setminus E}}\colon \RR^{\Gamma}\to\CC$ be defined by
\[f\otimes 1_{\RR^{\Gamma\setminus E}}(x)=f(x\big|_{E}).\]
 We say that $f\in S(\RR^{\Gamma})$ if there is a finite $E\subseteq \Gamma$ and a Schwartz function $f_{0}\colon \RR^{E}\to \CC$ so that
\[f=f_{0}\otimes 1_{\RR^{\Gamma\setminus E}}.\]
By standard Fourier analysis, there is a $\theta\in S(\RR^{E})$ so that
\[f(x)=\int_{\RR^{E}}\exp(2\pi it\cdot x)\,\theta(t)\,dt.\]

We prove that if we choose $x\in\RR^{d_{i}}$ with respect to the Gaussian measure on $p_{i}\RR^{d_{i}}$ then with high probability, the microstate $\phi^{(E)}_{x}$ will approximately preserve the measure $\mu_{p\delta_{e}}$ when integrated against Schwartz functions.

We need the following notation: if $\phi\in X^{d_{i}},\psi\in Y^{d_{i}}$ we define $\phi\otimes\psi\in (X\times Y)^{d_{i}}$ by
\[(\phi\otimes\psi)(j)=(\phi(j),\psi(j)).\]
\begin{lemma}\label{L:concentration}

 Let $\Gamma$ be a countable discrete sofic group with sofic approximation $\Sigma=(\sigma_{i}\colon\Gamma\to S_{d_{i}}).$ Let $p\in L_{\RR}(\Gamma)$ be an orthogonal projection. Fix a sequence of orthogonal projections $p_{i}\in M_{d_{i}}(\RR)$  such that
\[\|p_{i}-\sigma_{i}(p)\|_{2}\to 0.\]
Define the Gaussian measure $\nu_{i}$ on $p_{i}\ell^{2}_{\RR}(d_{i})$ by
\[d\,\nu_{i}(x)=e^{-\pi\|x\|_{\ell^{2}(d_{i})}^{2}}\,dx.\]
Here $dx$ is the Lebesgue measure on $p_{i}\ell^{2}_{\RR}(d_{i}).$ Let $\phi^{(E)}_{x}$ be defined as in Lemma \ref{L:generalmicrostates} for $X=\RR,x_{0}=0.$ Let $\mu_{p\delta_{e}}$ be defined as in Proposition \ref{P:cyclic}. Let $F\subseteq E$ be finite subsets of $\Gamma$ and $\delta>0.$ Then for any compact Hausdorff space $Y,$  sequence $\psi_{i}\in Y^{d_{i}},$ and $f\in S(\RR^{F})$ $, g\in C(Y),$
\[\nu_{i}\left(\left\{x\in p_{i}\RR^{d_{i}}:\left|\int f\otimes 1_{\RR^{\Gamma\setminus F}}\otimes  g\,d (\phi^{(E)}_{x}\otimes \psi_{i})_{*}(u_{d_{i}})-\int_{\RR^{\Gamma}}f\otimes 1_{\RR^{\Gamma\setminus F}}\,d\mu_{p\delta_{e}}\int g\,d(\psi_{i})_{*}(u_{d_{i}})\right|>\delta\right\}\right)\to 0.\]

\end{lemma}

\begin{proof}

Define $G\colon p_{i}\RR^{d_{i}}\to \RR$ by
\[G(x)=\int f\otimes 1_{\RR^{\Gamma\setminus E}}\otimes g\,d(\phi^{(E)}_{x}\otimes \psi_{i})_{*}(u_{d_{i}}).\]
We will show that
\begin{equation}\label{E:asymptoticexpecation}
\left|\int_{p_{i}\RR^{d_{i}}} G\,d\nu_{i}-\left(\int_{X}f\otimes 1_{\RR^{\Gamma\setminus E}}\,d\mu_{p\delta_{e}}\right)\left(\int_{Y}g\,d(\psi_{i})_{*}(u_{d_{i}})\right)\right|\to 0
\end{equation}
and
\begin{equation}\label{E:asymptoticmoments}
\left|\int_{p_{i}\RR^{d_{i}}}|G|^{2}\,d\nu_{i}-\left|\int_{X}f\otimes 1_{\RR^{\Gamma\setminus E}}\,d\mu_{p\delta_{e}}\right|^{2}\left|\int_{Y}g\,d(\psi_{i})_{*}(u_{d_{i}})\right|^{2}\right|\to 0.
\end{equation}
As
\[\left\|G-\int G\,d\nu_{i}\right\|_{L^{2}(d\nu_{i})}^{2}=\int_{p_{i}\RR^{d_{i}}}|G|^{2}\,d\nu_{i}-\left|\int_{p_{i}\RR^{d_{i}}}G\,d\nu_{i}\right|^{2},\]
the Lemma will then follow from Chebyshev's inequality.

Write
\[f(x)=\int_{\RR^{F}}\exp(2\pi i t\cdot x)\theta(t)\,dt\]
with $\theta\in S(\RR^{F}).$
Note that by Proposition \ref{P:cyclic} and by the fact that $\theta\in L^{1}(\RR^{F}),$
\begin{equation}\label{E:expectationcomputation}
\int_{\RR^{\Gamma}}f\otimes 1_{\RR^{\Gamma\setminus E}}\,d\mu_{p\delta_{e}}=\int_{\RR^{F}}\int_{\RR^{\Gamma}}\exp(2\pi i t\cdot x)\theta(t)\,d\mu_{p\delta_{e}}dt=\int_{\RR^{F}}\theta(t)\exp(-\pi\|\lambda(\check{t})p\|_{2}^{2})\,dt.
\end{equation}

We have
\begin{align*}
\int_{p_{i}\RR^{d_{i}}}G(x)\,d\nu_{i}(x)&=\frac{1}{d_{i}}\sum_{j=1}^{d_{i}}\int_{p_{i}\RR^{d_{i}}}\int_{\RR^{F}}\theta(t)\exp(2\pi i t\cdot\phi^{(E)}_{x}(j))g(\psi_{i}(j))\,dt\,d\nu_{i}(x)\\
&=\frac{1}{d_{i}}\sum_{j=1}^{d_{i}}\int_{\RR^{F}}\int_{p_{i}\RR^{d_{i}}}\theta(t)\exp(2\pi i t\cdot \phi^{(E)}_{x}(j))g(\psi_{i}(j))\,d\nu_{i}(x)\,dt
\end{align*}
the interchanges of integrals being valid as $g$ is bounded and $\theta\in L^{1}(\RR^{F}).$ If $E\supseteq F,$ then
\[t\cdot \phi^{(E)}_{x}(j)=\ip{\sigma_{i}(\check{t})x,e_{j}}_{\ell^{2}(d_{i})}.\]
Thus
\begin{align*}
\int_{p_{i}\RR^{d_{i}}}G(x)\,d\nu_{i}(x)&=\frac{1}{d_{i}}\sum_{j=1}^{d_{i}}\int_{\RR^{F}}\int_{p_{i}\RR^{d_{i}}}\theta(t)\exp(2\pi i \ip{\sigma_{i}(\check{t})x,e_{j}}_{\ell^{2}(d_{i},\RR)})g(\psi_{i}(j))\,d\nu_{i}(x)\,dt\\
&=\frac{1}{d_{i}}\sum_{j=1}^{d_{i}}\int_{\RR^{F}}\int_{p_{i}\RR^{d_{i}}}\theta(t)\exp(2\pi i\ip{x,p_{i}\sigma_{i}(\check{t})^{*}e_{j}}_{\ell^{2}(d_{i},\RR)})g(\psi_{i}(j))\,d\nu_{i}(x)\,dt\\
&=\frac{1}{d_{i}}\sum_{j=1}^{d_{i}}\int_{\RR^{F}}\theta(t)\exp(-\pi\|p_{i}\sigma_{i}(\check{t})^{*}e_{j}\|_{\ell^{2}(d_{i})}^{2})g(\psi(j))\,dt.
\end{align*}
Here we are using that
\[\int_{p_{i}\RR^{d_{i}}}\exp(2\pi i\ip{x,p_{i}\sigma_{i}(t)^{*}e_{j}}_{\ell^{2}(d_{i},\RR)})\,d\nu_{i}(x)=\exp(-\pi \|p_{i}\sigma_{i}(\check{t})^{*}e_{j}\|_{\ell^{2}(d_{i})}^{2}),\]
(this is an obvious generalization of Proposition 8.24 in \cite{Folland}).
By the preceding Lemma, there is a $C_{i}\subseteq \{1,\dots,d_{i}\}$ with
\begin{equation}\label{E:largeset}
\frac{|C_{i}|}{d_{i}}\to 1
\end{equation}
and
\[\limsup_{i\to\infty}\sup_{\substack{j\in C_{i}},\\ t\in\RR^{F}}\frac{\left|\|p_{i}\sigma_{i}(\check{t})^{*}e_{j}\|_{\ell^{2}(d_{i})}^{2}-\|\lambda(\check{t})p\|_{2}^{2}\right|}{\|t\|_{2}^{2}}=0.\]
As
\[\|\lambda(\check{t})p\delta_{e}\|_{2}^{2}=\tau(p\lambda(\check{t})^{*}\lambda(\check{t})p)=\tau(\lambda(\check{t})p\lambda(\check{t})^{*})=\|p\lambda(\check{t})^{*}\delta_{e}\|_{2}^{2},\]
 we have
\begin{equation}\label{E:arrghbound1}
\sup_{j\in C_{i}}\left|\int_{\RR^{F}}\theta(t)\exp(-\pi\|p_{i}\sigma_{i}(\check{t})^{*}e_{j}\|_{\ell^{2}(d_{i}) }^{2})\,dt-\int_{\RR^{F}}\theta(t)\exp(-\pi\|\lambda(\check{t})p\delta_{e}\|^{2})\,dt \right|\leq o(1)\int_{\RR^{F}}|\theta(t)|\|t\|_{\ell^{2}(F)}^{2}\,dt,
\end{equation}
where we use $o(1)$ for any expression that goes to zero as $i\to \infty.$
Additionally,
\begin{equation}\label{E:arrgghbound2}
\sup_{j\in \{1,\dots,d_{i}\}\setminus C_{i}}\left|\int_{\RR^{F}}\theta(t)\exp(-\pi\|p_{i}\sigma_{i}(\check{t})^{*}e_{j}\|_{\ell^{2}(d_{i})}^{2})\right|\leq \|\theta\|_{1}.
\end{equation}
Since
\[\int_{\RR^{F}}\theta(t)\|t\|_{\ell^{2}(F)}^{2}\,dt<\infty\]
 equations $(\ref{E:arrghbound1}),(\ref{E:arrgghbound2}),(\ref{E:largeset})$ and the fact that $g$ is bounded imply that
\[\left|\int_{p_{i}\RR^{d_{i}}}G(x)\,d\nu_{i}(x)-\int_{\RR^{F}}\theta(t)\exp(-\pi\|\lambda(\check{t})p\|_{2}^{2})\,dt\int g\,d\psi_{*}(u_{d_{i}})\right|\to 0.\]
By (\ref{E:expectationcomputation}) we have proved (\ref{E:asymptoticexpecation}).

	We now turn to the proof of $(\ref{E:asymptoticmoments}).$ By the same computations as above,
\begin{align*}
\int_{p_{i}\RR^{d_{i}}}|G|^{2}\,d\nu_{i}&=\frac{1}{d_{i}^{2}}\sum_{1\leq j,k\leq d_{i}}\int_{p_{i}\RR^{d_{i}}}\int_{\RR^{F}}\int_{\RR^{F}}\theta(t)\overline{\theta(s)}e^{2\pi i (\ip{\sigma_{i}(\check{t})x,e_{j}}_{\ell^{2}(d_{i})}-\ip{\sigma_{i}(\check{s})x,e_{k}}_{\ell^{2}(d_{i})})}g(\psi_{i}(j))\overline{g(\psi_{i}(k))}\,dt\,ds\,d\nu_{i}(x)\\
&=\frac{1}{d_{i}^{2}}\sum_{1\leq j,k\leq d_{i}}\int_{\RR^{F}}\int_{\RR^{F}}\theta(t)\overline{\theta(s)}e^{-\pi(\|p_{i}\sigma_{i}(\check{t})^{*}e_{j}-p_{i}\sigma_{i}(\check{s})^{*}e_{k}\|_{\ell^{2}(d_{i})}^{2})}g(\psi_{i}(j))\overline{g(\psi_{i}(k))}\,dt\,ds.
\end{align*}
Again the interchanges of integrals are valid as $\theta\in L^{1}(\RR^{F}).$ By the preceding Lemma, there are $A_{i}\subseteq\{1,\dots,d_{i}\}^{2}$ so that
\begin{equation}\label{E:otherbigset}
u_{d_{i}}\otimes u_{d_{i}}(A_{i})\to 1
\end{equation}
and
\[\lim_{i\to\infty}\sup_{\substack{(j,k)\in A_{i},\\ t,s\in\RR^{F}}}\frac{\left|\|p_{i}\sigma_{i}(\check{t})^{*}e_{j}-p_{i}\sigma_{i}(\check{s})^{*}e_{j}\|_{2}^{2}-\|\lambda(\check{t})p\delta_{e}\|_{2}^{2}-\|\lambda(\check{s})p\delta_{e}\|_{2}^{2}\right|}{(\|t\|_{\ell^{2}(F)}+\|s\|_{\ell^{2}(F)})^{2}}\,dt=0.\]
Thus
\begin{equation}\label{E:arghbound3}
\sup_{(j,k)\in A_{i}}\left|\int_{\RR^{F}}\int_{\RR^{F}}\theta(t)\overline{\theta(s)}e^{-\pi(\|p_{i}\sigma_{i}(\check{t})^{*}e_{j}-p_{i}\sigma_{i}(\check{s})^{*}e_{k}\|_{\ell^{2}(d_{i})}^{2})}\,dt\,ds-\left|\int_{\RR^{F}}\theta(t)e^{(-\pi\|p_{i}\sigma_{i}(\check{t})^{*}e_{j}\|_{\ell^{2}(d_{i})}^{2})}\,dt\right|^{2}\right|\to 0,
\end{equation}
as
\[\int_{\RR^{F}}\int_{\RR^{F}}|\theta(t)||\theta(s)|(\|t\|_{\ell^{2}(F)}+\|s\|_{\ell^{2}(F)})^{2}\,dt\,ds<\infty.\]
Additionally,
\begin{equation}\label{E:argbound4}
\sup_{(j,k)\in\{1,\dots,d_{i}\}^{2}\setminus (A_{i}\cap C_{i}\times C_{i})}\left|\int_{\RR^{F}}\int_{\RR^{F}}\theta(t)\overline{\theta(s)}e^{-\pi\left(\|p_{i}\sigma_{i}(t)^{*}e_{j}\|_{\ell^{2}(d_{i})}^{2}-2\ip{\sigma_{i}(s)p_{i}\sigma_{i}(t)^{*}e_{j},e_{k}}_{\ell^{2}(d_{i})}+\|p_{i}\sigma_{i}(t)e_{k}\|_{\ell^{2}(d_{i})}^{2}\right)}\right|\end{equation}
\[\leq \|\theta\|_{1}^{2}.\]

Equations $(\ref{E:arrghbound1}),(\ref{E:argbound4}),(\ref{E:otherbigset}),(\ref{E:largeset})$ and the fact that $g$ is bounded imply that
\[\left|\int_{p_{i}\RR^{d_{i}}}|G|^{2}\,d\nu_{i}-\left|\int_{\RR^{F}}\theta(t)\exp(-\pi\|\lambda(\check{t})p\delta_{e}\|_{2}^{2})\,dt\right|^{2}\cdot \frac{1}{d_{i}^{2}}\sum_{1\leq j,k\leq d_{i}}g(\psi_{i}(j))\overline{g(\psi_{i}(k))}\right|\to 0.\]
Since
\[\frac{1}{d_{i}^{2}}\sum_{1\leq j,k\leq d_{i}}g(\psi_{i}(j))\overline{g(\psi_{i}(k)}=\left|\int g\,d\psi_{*}(u_{d_{i}})\right|^{2}\]
we have proved $(\ref{E:asymptoticmoments}).$

\end{proof}

\begin{lemma}Let $A$ be an infinite set.   Then the space $S(\RR^{A})$ generates $C_{b}(\RR^{A})$ in the sense of Definition \ref{D:generating}.

\end{lemma}

\begin{proof}

	 Let $\varepsilon>0,f\in C_{b}(\RR^{A})$ and $K\subseteq \RR^{A}$ compact. We say that a function $\psi\colon K\to \CC$ depends upon finitely many coordinates if there is a finite $E\subseteq A$ so that if $x,y\in K$ and
\[x(a)=y(a)\mbox{ for all $a\in E$}\]
then
\[f(x)=f(y).\]

By the Stone-Weierstrass Theorem there is a function $\psi\colon K\to \CC$ depending upon finitely many coordinates so that
\[\|\psi\|_{C(K)}\leq \|f\big|_{K}\|_{C(K)}\leq \|f\|_{C_{b}(X)}\]
and
\[\|\psi-f\big|_{K}\|_{C(K)}<\varepsilon.\]
Let $E\subseteq A$ be such that if $x,y\in K$ and
\[x(a)=y(a)\mbox{ for all $a\in E$}\]
then $\psi(x)=\psi(y).$ Let
\[K_{E}=\{x\in \RR^{E}:\mbox{ there is a $y\in K$ with $x(a)=y(a)$  for all $a\in E$}\}.\]
 Then $K_{E}$ is a compact, being the continuous image of $K$ under the projection map $\RR^{A}\to \RR^{E}.$ There is a well-defined function $\widetidle{\psi}\colon K_{E}\to\CC$ such that
\[\widetilde{\psi}(x)=\psi(x')\]
whenever $x'\in \RR^{\Gamma}$ has $x'(g)=x(g)$ for all $g\in E.$ It is well known that there is a $\widetilde{\phi}\in C_{c}^{\infty}(\RR^{E})$ so that
\[\|\widetidle{\phi}\big|_{K_{E}}-\widetidle{\psi}\big\|_{C(K_{E})}<\varepsilon\]
\[\|\widetilde{\phi}\|_{C_{b}(\RR^{E})}\leq \|f\|_{C_{b}(\RR^{A})}.\]
Let
\[\phi=\widetilde{\phi}\otimes 1_{\RR^{A\setminus E}}.\]
Since $\widetilde{\phi}$ is a Schwartz function, we have that $\phi\in S(\RR^{A}).$  Finally,
\[\|\phi\big|_{K}-f\big|_{K}\|_{C(K)}\leq \|\phi\big|_{K}-\psi\|_{C(K)}+\|\psi-f\big|_{K}\|_{C(K)}<2\varepsilon.\]
This completes the proof.

\end{proof}

\begin{theorem}\label{T:GaussianproductSection6} Let $\Gamma$ be a countable discrete sofic group with sofic approximation $\Sigma=(\sigma_{i}\colon\Gamma\to S_{d_{i}}).$ Let $\rho\colon\Gamma\to \mathcal{O}(\mathcal{H})$  be an orthogonal representation on a real, separable Hilbert space $\mathcal{H}$ with $\rho\ll \lambda_{\Gamma,\RR}.$ Let $\Gamma\actson (X_{\rho},\mu_{\rho})$ be the corresponding Gaussian action. Let $(Y,\nu)$ be a standard probability space and $\Gamma\actson (Y,\nu)$ a measure-preserving action with $h_{\Sigma,\nu}(Y,\Gamma)\geq 0.$ Then $h_{\Sigma,\mu_{\rho}\otimes \nu}(X_{\rho}\times Y,\mu_{\rho}\otimes\nu)=\infty.$
\end{theorem}

\begin{proof}
 We shall first reduce to the case that $\rho$ is cyclic, i.e there is a vector $\xi\in\mathcal{H}$ so that
\[\mathcal{H}=\overline{\Span\{\rho(g)\xi:g\in\Gamma\}},\]
so suppose we can prove the Theorem in the cyclic case. Let
\[\rho=\rho_{1}\oplus \rho_{2},\]
where $\rho_{1}$ is cyclic and $\rho_{2}\ll \lambda_{\Gamma,\RR}.$ Since we are assuming the theorem in the cyclic case and
\[\Gamma\actson (X_{\rho}\times Y,\mu_{\rho}\otimes \nu)\cong \Gamma\actson (X_{\rho_{1}}\times X_{\rho_{2}}\times Y,\mu_{\rho_{1}}\otimes \mu_{\rho_{2}}\otimes \nu)\]
it suffices to show that
\begin{equation}\label{E:reductionSection6asdlghlajh}
h_{\Sigma,\mu_{\rho_{2}}\otimes \nu}(X_{\rho_{2}}\times Y,\Gamma)\geq 0,
\end{equation}
as $\rho_{1}$ is cyclic. Since $\mathcal{H}$ is separable, we may write
\[\rho_{2}=\bigoplus_{n=1}^{\infty}\rho_{2,n}\]
where $\rho_{2,n}$ is cyclic. For $n\in\NN,$ let
\[\rho_{2,\leq N}=\bigoplus_{n=1}^{N}\rho_{2,n}.\]
By the cyclic case and induction, we have
\[h_{\Sigma,\mu_{\rho_{2,\leq N}}\otimes \nu}(X_{\rho_{2,\leq N}}\times Y,\Gamma)=\infty\]
for any $N\in\NN.$ Using that
\[\Gamma\actson (X_{\rho_{2}}\times Y,\mu_{\rho_{2}}\otimes \nu)\cong \varprojlim\Gamma\actson(X_{\rho_{2,\leq N}}\times Y,\mu_{\rho_{2,\leq N}}\otimes\nu),\]
 it is not hard to show (\ref{E:reductionSection6asdlghlajh}).

We use the preceding Lemma and Theorem \ref{T:generating}.  By Proposition \ref{P:cyclic} we may regard $\Gamma\actson (X_{\rho},\mu_{\rho})$ as $\RR^{\Gamma}$ with the measure $\mu_{p\delta_{e}}$ defined by
\[\int_{\RR^{\Gamma}}\exp(2\pi i t\cdot x)\,d\mu_{p\delta_{e}}(x)=\exp(-\pi\|\lambda(t)p\delta_{e}\|_{2}^{2})\]
for $t\in c_{c}(\Gamma,\RR)$ and some orthogonal projection $p\in L_{\RR}(\Gamma).$  Extend $\Sigma$ to an approximation sequence of $L(\Gamma)$ still denoted $\sigma_{i}\colon L(\Gamma)\to M_{d_{i}}(\CC).$ By Proposition \ref{P:realextension}, there exists a sequence $p_{i}$ of orthogonal projections in $M_{d_{i}}(\RR)$ so that
\[\|p_{i}-\sigma_{i}(p)\|_{2}\to 0.\]
Let $\nu_{i}$ be the Gaussian measure on $p_{i}\RR^{d_{i}}$ defined by
\[d\nu_{i}(x)=e^{-\pi\|x\|_{\ell^{2}(d_{i})}^{2}}\,dx.\] Choosing a compact model, we may assume that $Y$ is a compact metrizable space and that $\Gamma\actson Y$ by homeomorphisms. Let $\Delta_{Y}$ be a compatible metric on $Y.$ Let $\Delta_{\RR^{\Gamma}}$ be the dynamically pseudometric on $\RR^{\Gamma}$ defined by
\[\Delta_{\RR^{\Gamma}}(x,y)=\min(|x(e)-y(e)|,1).\]
We shall use the generating set $\mathcal{L}=\{f\otimes g:f\in S(\RR^{\Gamma}),g\in C(Y)\}$. We use the dynamically generating pseudometric on $\RR^{\Gamma}\times Y$ defined by
\[\Delta((x_{1},y_{1}),(x_{2},y_{2}))=\Delta_{\RR^{\Gamma}}(x_{1},x_{2})+\Delta_{Y}(y_{1},y_{2}).\]
Let $\varepsilon,\delta,\eta>0$ and finite $F\subseteq\Gamma$,$L\subseteq\mathcal{L}$ be given.

Let $e=g_{1}$,$g_{2},\dots$ be an enumeration of the elements of $\Gamma.$ Inductively find positive real numbers $M_{1},M_{2},\dots$ so that
\[\mu_{p\delta_{e}}\left(\{x\in \RR^{\Gamma}:~|x(g_{j})|\leq M_{j},j=1,\dots,l\}\right)\geq 1-\sum_{j=1}^{l}2^{-j}\eta.\]
Define $M\colon \Gamma\to (0,\infty)$ by $M(g_{j})=M_{j}.$ Set
\[K_{M}=\{x\in \RR^{\Gamma}:|x(g)|\leq M(g)\mbox{ for all $g\in\Gamma$}\}\]
and note that $K_{M}$ is compact.
Let $U$ be an open neighborhood of $K_{M}$ in $\RR^{\Gamma}.$

Let $L_{1}\subseteq S(\RR^{\Gamma}),L_{2}\subseteq C(Y)$ be finite sets so that
\[L\subseteq \{f\otimes g:f\in L_{1},g\in L_{2}\}.\]
Since
\[h_{\Sigma,\nu}(Y,\Gamma)\geq 0,\]
for infinitely many $i,$ there is a $\psi_{i}\in \Map_{\nu}(\Delta_{Y},F,\delta,L_{2},\sigma_{i}).$  For these $i,$ by Lemma \ref{L:concentration}, there is a finite set $E$ of $\Gamma$ containing the identity and an $B_{i}\subseteq p_{i}\RR^{d_{i}}$ so that
\[\nu_{i}(B_{i})\to 1\]
\[\phi^{(E)}_{x}\otimes \psi_{i}\in \Map_{\mu\otimes \nu}(\Delta,F,\delta,L,\sigma_{i})\]
for all $x\in B_{i}.$ We will also assume that
\[U\supseteq \pi_{E}(K_{M})\times \RR^{\Gamma\setminus E}.\]
This may be done by enlarging $E,$  as $K_{M}$ is compact. Choose $f\in S(\RR^{E})$ so that
\[\int_{\RR^{E}}f\,d(\pi_{E})_{*}\mu_{\rho}\geq -\eta+(\pi_{E})_{*}\mu(\pi_{E}(K_{M}))\]
and
\[0\leq f\leq \chi_{\pi_{E}(K_{M})}.\]
We may assume that
\[\phi^{(E)}_{x}\otimes \psi_{i}\in \Map_{\mu\otimes \nu}(\Delta,F,\delta,L\cup\{f\otimes 1_{\RR^{\Gamma\setminus E}\otimes 1}\},\sigma_{i}),\]
for all $x\in B_{i}.$
Note that in this case
\begin{align*}
(\phi^{(E)}_{x}\otimes \psi_{i})_{*}(u_{d_{i}})(U\times Y)&\geq (\phi^{(E)}_{x}\otimes \psi_{i})_{*}(u_{d_{i}})(\pi_{E}(K_{M})\times \RR^{\Gamma\setminus E}\times Y)\\
&\geq -\delta+\int_{\RR^{\Gamma}}f\otimes 1_{\RR^{\Gamma\setminus E}}\,d\mu_{p}\\
&\geq -\delta-\eta+(\pi_{E})_{*}\mu(\pi_{E}(K_{M}))\\
&\geq-\delta-2\eta+1.
\end{align*}
So if $\delta<\eta$ we have
\[\phi^{(E)}_{x}\otimes \psi_{i}\in \Map_{\mu\otimes\nu}^{U\times Y,3\eta}(\Delta,F,\delta,L\cup\{f\otimes 1_{\RR^{\Gamma\setminus E}}\otimes 1\},\sigma_{i}).\]
Thus
\[h_{\Sigma,\mu_{p}\otimes \nu}^{U\times Y,\eta}(\Delta,\varepsilon,F,\delta,L\cup\{f\otimes 1_{\RR^{\Gamma\setminus E}}\otimes 1\})\geq \limsup_{i\to\infty}\frac{1}{d_{i}}\log S_{\varepsilon}(B_{i},\Delta_{2}).\]

	Suppose that $S\subseteq B_{i}$ is $\varepsilon$-dense. Then for every $y\in B_{i},$ there exists $x\in S$ and a $C\subseteq\{1,\dots,d_{i}\}$ so that
\[|x(j)-y(j)|<\sqrt{\varepsilon},\mbox{ for $j\in C$},\]
\[|C|\geq (1-\sqrt{\varepsilon})d_{i}.\]
Using $m$ for Lebesgue measure,
\begin{align*}
\nu_{i}(B)&\leq \sum_{\substack{x\in S,\\ A\subseteq \{1,\dots,d_{i}\},\\ |A|\geq (1-\sqrt{\varepsilon})d_{i}}}\nu_{i}((x_{A}+\sqrt{\varepsilon}\Ball(\ell^{2}_{\RR}(A,u_{A})))\times \RR^{A^{c}})\\
&\leq  \sum_{\substack{x\in S,\\ A\subseteq \{1,\dots,d_{i}\},\\ |A|\geq (1-\sqrt{\varepsilon})d_{i}}}m(\sqrt{\varepsilon}\Ball(\ell^{2}_{\RR}(A,u_{A})))\\
&=\sum_{\substack{x\in S,\\ A\subseteq \{1,\dots,d_{i}\},\\ |A|\geq (1-\sqrt{\varepsilon})d_{i}}}\varepsilon^{|A|/2}m(\Ball(\ell^{2}_{\RR}(A,u_{A})))
\end{align*}
By \cite{Folland} Corollary 2.55 and Stirling's Formula, there is a $R>0$ so that
\[m(\Ball(\ell^{2}_{\RR}(A,u_{A})))\leq R^{|A|/2}\left(\frac{d_{i}}{|A|}\right)^{|A|/2}.\]
Set
\[H(t)=\begin{cases}
0,&\textnormal{ if $t=0,1$}\\
-t\log(t)-(1-t)\log(1-t),&\textnormal{ for $0<t<1$.}
\end{cases}\]
Then for some $D>0$ we have
\begin{align*}
\nu_{i}(B)&\leq\sum_{\substack{x\in S,\\ A\subseteq \{1,\dots,d_{i}\},\\ |A|\geq (1-\sqrt{\varepsilon})d_{i}}}\varepsilon^{\frac{d_{i}}{2}}R^{|A|/2}\left(\frac{d_{i}}{|A|}\right)^{|A|/2}\\
&\leq |S|D\sqrt{\varepsilon}d_{i} R^{d_{i}/2}\exp(H(\sqrt{\varepsilon})d_{i})\left(\frac{\varepsilon}{(1-\sqrt{\varepsilon})}\right)^{\frac{d_{i}}{2}},
\end{align*}
the last line following by Stirling's Formula.
Thus
\[\limsup_{i\to\infty}\frac{1}{d_{i}}\log |S|\geq \frac{1}{2}\log\left(\frac{1-\sqrt{\varepsilon}}{\varepsilon}\right)-\frac{1}{2}\log R-H(\sqrt{\varepsilon}).\]
So
\[h_{\Sigma,\mu_{p}\otimes \nu}^{U\times Y,3\eta}(\Delta,\varepsilon,F,\delta,L)\geq h_{\Sigma,\mu_{p}\otimes \nu}^{U\times Y,3\eta}(\Delta,\varepsilon,F,\delta,L\cup\{f\otimes \RR^{\Gamma\setminus E}\otimes 1\})\geq \frac{1}{2}\log\left(\frac{1-\sqrt{\varepsilon}}{\varepsilon}\right)-\frac{1}{2}\log R-H(\sqrt{\varepsilon}),\]
taking the infimum over all $U\supseteq K_{M},$ then the supremum over all $K_{M},$ then the infimum over all  $F,\delta,\eta$ and $L\subseteq\mathcal{L}$ we find that
\[h_{\Sigma,\mu_{p\delta_{e}}\otimes\nu}(\Delta,\varepsilon,\mathcal{L})\geq \frac{1}{2}\log\left(\frac{1-\sqrt{\varepsilon}}{\varepsilon}\right)-\frac{1}{2}\log R-H(\sqrt{\varepsilon}) .\]
If we now let $\varepsilon\to 0$ we see that
\[h_{\Sigma,\mu_{p\delta_{e}}\otimes \nu}(X\times Y,\Gamma)=\infty.\]

\end{proof}

We now prove a general formula for the entropy of Gaussian actions.

\begin{cor}\label{C:GaussianEntropyFormula} Let $\Gamma$ be a countable discrete sofic group with sofic approximation $\Sigma=(\sigma_{i}\colon \Gamma\to S_{d_{i}}).$ Let $\rho\colon \Gamma\to \mathcal{O}(\mathcal{H})$ be an orthogonal representation. By Proposition \ref{P:LebesgueDecompositionOR} write $\rho=\rho_{1}\oplus \rho_{2}$ where $\rho_{1}\ll \lambda_{\Gamma,\RR}$ and $\rho_{2}\perp \lambda_{\Gamma,\RR}.$ Let $\Gamma\actson (X_{\rho},\mu_{\rho}),\Gamma\actson (X_{\rho_{j}},\mu_{\rho_{j}}),j=1,2$  be the corresponding Gaussian actions. Then:

(i): $h_{\Sigma,\mu_{\rho_{2}}}(X_{\rho_{2}},\mu_{\rho_{2}})\in \{0,-\infty\},$

(ii):
\[h_{\Sigma,\mu_{\rho}}(X_{\rho_{1}},\mu_{\rho_{1}})=\begin{cases}
-\infty, &\mbox{ if $h_{\Sigma,\mu_{\rho_{2}}}(X_{\rho_{2}},\mu_{\rho_{1}})=-\infty,$}\\
0 , &\mbox{ if $\rho_{1}=0$ and $h_{\Sigma,\mu_{\rho_{2}}}(X_{\rho_{2}},\mu_{\rho_{2}})=0,$}\\
\infty, &\mbox{ if $\rho_{1}\ne 0$ and $h_{\Sigma,\mu_{\rho_{2}}}(X_{\rho_{2}},\mu_{\rho_{2}})=0.$}
\end{cases}\]
\end{cor}

\begin{proof}

Statement $(i)$ is just a direct corollary of Corollary \ref{C:zerogaussian} and the fact that sofic entropy is always nonnegative or $-\infty.$  The first case of statement $(i)$ follows from the general fact that if $\Gamma\actson (X,\mu),\Gamma\actson (Y,\nu)$ are two measure-preserving actions on standard probability spaces, then
\[h_{\Sigma,\mu}(X,\Gamma)=-\infty\]
implies that
\[h_{\Sigma,\mu\otimes \nu}(X\times Y,\Gamma)=-\infty.\]
The second case of statement $(ii)$ is also just a special case of statement $(i).$ The last case of statement $(ii)$ follows from Theorem \ref{T:GaussianproductSection6}.

\end{proof}

We give some examples to show that $h_{\Sigma,X_{\rho}}(X_{\rho},\mu_{\rho})$ can be $-\infty.$ Let $\Gamma$ be a countable discrete sofic group with sofic approximation $\Sigma=(\sigma_{i}\colon \Gamma\to S_{d_{i}}).$ We say that $\Sigma$ is \emph{ergodic} if whenever $A_{i}\subseteq \{1,\dots,d_{i}\}$  are such that
\[\lim_{i\to \infty}u_{d_{i}}(A_{i}\Delta\sigma_{i}(g)A_{i})=0\mbox{ for all $g\in \Gamma$}\]
then
\[\lim_{i\to\infty}u_{d_{i}}(A_{i})(1-u_{d_{i}}(A_{i}))=0.\]

The following is a folklore result and we include a proof for completeness.

\begin{proposition} Let $\Gamma$ be a countable discrete sofic group with an ergodic sofic approximation $\Sigma.$ Then, if $\Gamma\actson (X,\mu)$ is a nonergodic probability measure-preserving action on a standard probability space, we have
\[h_{\Sigma,\mu}(X,\Gamma)=-\infty.\]
\end{proposition}

\begin{proof}
Let $\Sigma=(\sigma_{i}\colon\Gamma\to S_{d_{i}}).$  Let $E\subseteq X$ be a $\Gamma$-invariant set with $0<\mu(E)<1.$ Let $\alpha\colon X\to \{0,1\}$ be the finite observable given by
\[\alpha(x)=\chi_{E}(x).\]
Let $\varepsilon>0,$ by a diagonalization argument it is easy to see that there is a $\delta>0$  and a finite $F\subseteq\Gamma$ containing $e$ so that if  $B_{i}\subseteq \{1,\dots,d_{i}\}$ has
\begin{equation}\label{E:asymptoticerogidcity}
\limsup_{i\to \infty}\max_{g\in F}u_{d_{i}}(B_{i}\Delta\sigma_{i}(g)B_{i})<\delta,
\end{equation}
then
\begin{equation}\label{E:asymptoticmeasurezero}
\limsup_{i\to \infty}u_{d_{i}}(B_{i})(1-u_{d_{i}}(B_{i}))< \varepsilon.
\end{equation}
Set $\widetidle{E}=\{a\in \{0,1\}^{F}:a(e)=1\}.$ If $\delta$ is sufficiently small, and $\phi\in\AP(\alpha,F,\delta',\sigma_{i})$ then since $\mu(gE\Delta E)=0,$ we have
\[\max_{g\in F}u_{d_{i}}(\sigma_{i}(g)\phi^{-1}(\widetilde{E})\Delta\phi(\widetilde{E}))<\delta\]
and we always have
\[\mu(E)+\delta'\geq u_{d_{i}}(\phi^{-1}(\widetilde{E}))\geq \mu(E)-\delta'.\]
If $\delta',\varepsilon$ are sufficiently small, then
\[u_{d_{i}}(\phi^{-1}(\widetilde{E}))(1-u_{d_{i}}(\phi^{-1}(\widetilde{E})))>\varepsilon.\]
Thus we see from $(\ref{E:asymptoticerogidcity}),(\ref{E:asymptoticmeasurezero})$ that $\AP(\alpha,F,\delta',\sigma_{i})=\varnothing.$ Thus $h_{\Sigma,\mu}(X,\Gamma)=-\infty,$ using Kerr's definition of entropy via partitions.

\end{proof}

Combining with Theorem 2.8 of \cite{PetersonSinclair}   we have the following.

\begin{cor} Let $\Gamma$ be a countable discrete sofic group with an ergodic sofic approximation $\Sigma.$ Let $\rho\colon\Gamma\to \mathcal{O}(\mathcal{H})$ be an orthogonal representation which is not weakly mixing (e.g. $\rho$ could be compact). Then if $\Gamma\actson (X,\mu)$ is the associated Gaussian action,
\[h_{\Sigma,\mu}(X,\Gamma)=-\infty.\]

\end{cor}

We mention an example of an ergodic sofic approximation when $\Gamma=\FF_{2}$ is the free group on two generators $a,b.$ Here we can choose a sofic approximation randomly. Namely let $\phi=(\phi_{1},\phi_{2})\in S_{n}^{2}$ be chosen uniformly at random, and let
\[\sigma_{\phi}\colon\FF_{2}\to S_{n}\]
be the unique homomorphism so that $\sigma_{\phi}(a)=\phi_{1},\sigma_{\phi}(b)=\phi_{2}.$ It is known that with high probability $\sigma_{\phi}$ is a sofic approximation (see \cite{Nica}, as well as \cite{KerrCPE} Lemma 3.1). It is also known from the theory of expanders that with high probability $\sigma_{\phi}$ is a ergodic sofic approximation (see \cite{Fried}, and the remarks in Section 5 of \cite{Pan2}). If we take an orthogonal representation $\rho\colon\FF_{2}\to \mathcal{O}(\mathcal{H})$ which is not weakly mixing (e.g. take $\mathcal{H}$ to be finite-dimensional) then we have
\[h_{\Sigma,\mu_{\rho}}(X_{\rho},\Gamma)=-\infty.\]
%

\end{document}